\documentclass{amsart}

\usepackage[all]{xy}

\usepackage{hyperref}

\usepackage{cleveref}

\usepackage{amssymb}

\theoremstyle{plain}
\newtheorem{theorem}[subsection]{Theorem}
\newtheorem{proposition}[subsection]{Proposition}
\newtheorem{lemma}[subsection]{Lemma}

\theoremstyle{definition}
\newtheorem{definition}[subsection]{Definition}

\newcommand{\ad}{\mathrm{ad}}
\newcommand{\Aut}{\mathrm{Aut}}
\newcommand{\Bun}{\mathrm{Bun}}
\newcommand{\F}{\mathbb{F}}

\newcommand{\Gal}{\mathrm{Gal}}
\newcommand{\Hom}{\mathrm{Hom}}

\newcommand{\Q}{\mathbb{Q}}
\newcommand{\Rep}{\mathrm{Rep}}

\newcommand{\Spec}{\textrm{Spec}}
\newcommand{\Z}{\mathbb{Z}}

\begin{document}

\title{Reductive group schemes over the Fargues-Fontaine curve}
\author{Johannes Ansch\"{u}tz}
\email{ja@math.uni-bonn.de}
\date{\today}

\begin{abstract}
For an arbitrary non-archimedean local field we classify reductive group schemes over the corresponding Fargues-Fontaine curve by group schemes over the category of isocrystals. We then classify torsors under such reductive group schemes by a generalization of Kottwitz' set $B(G)$. In particular, we extend a theorem of Fargues on torsors under constant reductive groups to the case of equal characteristic. 
\end{abstract}

\maketitle

\section{Introduction}

Let $E$ be a non-archimedean local field and let $G/E$ be a reductive group over $E$. If $E/\Q_p$ is $p$-adic L.\ Fargues classified $G$-torsors on the corresponding Fargues-Fontaine curve $X_E$ (cf.\ \cite{fargues_g_torseurs_en_theorie_de_hodge_p_adique}) associated to $E$ (and an algebraically closed perfectoid extension $F$ of the residue field $\F_q$ of $E$). His result is phrased in terms of R.\ Kottwitz' set $B(G)$ associated with $G$ (cf.\ \cite{kottwitz_isocrystals_with_additional_structure_I}). For a general non-archimedean local field $E$ let 
$$
L:=\widehat{E^\mathrm{un}}
$$ 
be the completion of the maximal unramified extension $E^{\mathrm{un}}$ of $E$ and let 
$$
\varphi_L\colon L\to L
$$ be the Frobenius of $L$. In \cite{kottwitz_isocrystals_with_additional_structure_I} (cf.\ \cite{kottwitz_bg_for_all_local_and_global_fields}) R.\ Kottwitz defined
$$
B(G):=G(L)/{\varphi_L-\mathrm{conjugacy}}
$$
as the set of $\varphi_L-$conjugacy classes in $G(L)$.
Our first main theorem is the classification of $G$-torsors on the Fargues-Fontaine curve $X_E$ by this set $B(G)$ generalizing L.\ Fargues' result for $E/\Q_p$ $p$-adic.

\begin{theorem}[cf.\ \Cref{theorem: fargues theorem}]
\label{theorem: fargues theorem introduction}
There exists a canonical bijection
$$
B(G)\cong H^1_{\acute{e}t}(X_E,G).
$$
\end{theorem}

We mention that this theorem implies the computation of the Brauer group of $E$ (cf.\ \cite[Th\'eor\`eme 2.6.]{fargues_g_torseurs_en_theorie_de_hodge_p_adique} resp.\ \Cref{theorem: brauer group of the curve}) and that it is a starting point for L.\ Fargues' program to geometrize the local Langlands correspondence over $E$ (cf.\ \cite{fargues_geometrization_of_the_local_langlands_correspondence_overview}).
Similar to \cite{fargues_g_torseurs_en_theorie_de_hodge_p_adique} we present in \Cref{section: applications} applications of \Cref{theorem: fargues theorem introduction} to the \'etale and flat cohomology of $X_E$ with locally constant torsion coefficents.

In \cite[Remarque 1.2.]{fargues_g_torseurs_en_theorie_de_hodge_p_adique} L.\ Fargues asked the question how to generalize \Cref{theorem: fargues theorem introduction} about $G$-torsors on $X_E$ to the case where $G$ is no longer assumed to be constant, but an arbitrary reductive group scheme over $X_E$. Examples of possibly non-constant reductive group schemes over the Fargues-Fontaine curve $X_E$ can be constructed as follows. Let $\varphi-\mathrm{Mod}_L$ be the category of isocrystals over $L$ and let $\mathcal{O}_{\mathbb{G}}$ be a Hopf algebra in $\varphi-\mathrm{Mod}_L$ (such a Hopf algebra is also called an affine group scheme $\mathbb{G}$ over $\varphi-\mathrm{Mod}_L$ (cf.\ \Cref{definition: group scheme over isocrystals})). After fixing an embedding $\bar{\F}_q\subseteq F$ there exists a canonical tensor functor (cf.\ \cite[Section 8.2.3.]{fargues_fontaine_courbe_et_fibres_vectories_en_theorie_de_hodge_p_adique} resp.\ \cite{hartl_pink_vector_bundles_with_frobenius_structure})
$$
\mathcal{E}(-)\colon \varphi-\mathrm{Mod}_L\to \Bun_{X_E}
$$
from the category $\varphi-\mathrm{Mod}_L$ to the category of vector bundles on $X_E$. The image $\mathcal{E}(\mathcal{O}_{\mathbb{G}})$ of $\mathcal{O}_{\mathbb{G}}$ is then the Hopf algebra associated with a flat group scheme $\mathcal{G}$ over $X_E$. We call this group scheme $\mathcal{G}$ the ``group scheme associated with the affine group scheme $\mathbb{G}$ over $\varphi-\mathrm{Mod}_L$''.
Our first theorem about reductive group schemes over the Fargues-Fontaine curve is the following classification result.

\begin{theorem}[cf.\ \Cref{theorem: classification of reductive group schemes}]
\label{theorem: classification of reductive group schemes introduction}
Every reductive group scheme over the Fargues-Fontaine curve is associated with a, necessarily reductive, group scheme over $\varphi-\mathrm{Mod}_L$.
\end{theorem}

The proof of \Cref{theorem: classification of reductive group schemes introduction} is straightforward, based on the classification of torsors for reductive groups (cf.\ \Cref{theorem: fargues theorem introduction}) and the geometric simply-connectedness of $X_E$ (cf.\ \Cref{theorem: etale fundamental group of the curve}).

However, reductive group schemes over $X_E$ are rather close to being constant (cf.\ \Cref{lemma: reductive group schemes over punctured curve}). For example, they become constant after removing a closed point of $X_E$ (cf.\ \Cref{lemma: reductive group schemes over punctured curve}).

Based on \Cref{theorem: classification of reductive group schemes introduction} we can define for a reductive group scheme $\mathcal{G}$ over $X_E$ a natural candidate of a set $B(\mathcal{G})$ generalizing Kottwitz' definition in the constant case (cf.\ \Cref{definition: b of t} and \Cref{proposition: explicit description of b of g}). Namely, if $\mathcal{G}$ is associated with the group scheme $\mathbb{G}$ we can set $B(\mathcal{G})$ as the set of $\varphi_L$-conjugacy classes in $G(L)$ where $G$ is the reductive group over $L$ associated with the Hopf algebra $\mathcal{O}_{\mathbb{G}}$.\footnote{Later we call this set $B(\mathbb{G})$.}

Our second theorem about reductive group schemes over $X_E$ is then the following classification of torsors.

\begin{theorem}[cf.\ \Cref{theorem: classification of torsors general case}]
\label{theorem: classification of torsors general case introduction}
For a reductive group scheme $\mathcal{G}$ over $X_E$ there is a canonical bijection 
$$
B(\mathcal{G})\cong H^1_{\acute{e}t}(X_E,\mathcal{G}).
$$
\end{theorem}

The proof of \Cref{theorem: fargues theorem introduction} and \Cref{theorem: classification of torsors general case introduction} is based on an excessive use of the Tannakian formalism of $\Q$-graded and $\Q$-filtered fiber functors, which is an easy adaptation of \cite{ziegler_graded_and_filtered_fiber_functors} (cf.\ \Cref{section: q-filtered fiber functors}, and the classification of vector bundles on $X_E$ (cf.\ \Cref{theorem: classification of vector bundles}), especially the fact that for a semistable vector bundle $\mathcal{E}$ of positive slope the cohomology group 
$$
H^1(X_E,\mathcal{E})=0
$$
vanishes. The crucial input in our proof is the fact that sending a vector bundle to its Harder-Narasimhan filtration defines a fully faithful, unfortunately non-exact, tensor functor
$$
\mathrm{HN}\colon \Bun_{X_E}\to \mathrm{Fil}^\Q\Bun_{X_E}
$$
from ordinary vector bundles on $X_E$ to $\Q$-filtered vector bundles on $X_E$.

Finally, we discuss uniformization results (cf.\ \Cref{section: uniformization}) for $\mathcal{G}$-torsors under arbitrary reductive group schemes $\mathcal{G}$ over $X_E$ generalizing the known case that $G$ is constant, quasi-split over $E$ with $E/\Q_p$ $p$-adic (cf.\ \cite[Th\'eor\`eme 7.2.]{fargues_g_torseurs_en_theorie_de_hodge_p_adique}).

\begin{theorem}[cf.\ \Cref{theorem: uniformization}]
\label{theorem: uniformization introduction} 
Let $\mathcal{G}$ be a reductive group scheme over $X_E$. Then every $\mathcal{G}$-torsor becomes trivial after removing a closed point of $X_E$.  
\end{theorem}

\subsection*{Acknowledgement} The author wants to thank Laurent Fargues, Jochen Heinloth, Michael Rapoport, Peter Scholze and Torsten Wedhorn for answering several questions related to this paper. Especially, the hint of Michael Rapoport to \cite[Theorem 5.3.1.]{dat_orlik_rapoport_period_domains} lead to the full proof of \Cref{theorem: fargues theorem} in the equal characteristic case.  

\section{$\Q$-filtered fiber functors}
\label{section: q-filtered fiber functors}

In this section we want to extend the results of \cite{ziegler_graded_and_filtered_fiber_functors} about $\Z$-filtered fiber functors to $\Q$-filtered fiber functors. Filtrations on functors indexed by $\mathbb{Q}$ are also discussed in \cite{dat_orlik_rapoport_period_domains}, e.g., in Chapter IV.2, or \cite{cornut_filtrations_and_buildings}.

We follow the definitions and notations of \cite[Chapter 2]{ziegler_graded_and_filtered_fiber_functors}, but for us fiber functors will take values in the category of locally free sheaves of finite ranks and not in the category of arbitrary quasi-coherent modules.\footnote{which is not a restriction as fiber functors take their image in vector bundles}

Let $k$ be a field and let $\mathcal{T}$ be a Tannakian category over $k$.
Let $\Gamma$ be a totally ordered abelian group. Mainly, we will be interested in the case that $\Gamma\subseteq\Q$ is a subgroup. 
We start by recalling the definition of a $\Gamma$-graded fiber functor (cf.\ \cite[Chapitre IV.1]{saavedra_rivano_categories_tannakiennes}).
For this we denote by $\mathrm{Gr}^\Gamma\Bun_S$ the tensor category of $\Gamma$-graded vector bundles on a $k$-scheme $S$.
Equivalently, if $D_{\Gamma,S}$ denotes the constant multiplicative group (over $S$) with character group $\Gamma$, then the category $\mathrm{Gr}^\Gamma\Bun_S$ is equivalent to the category of representations of $D_{\Gamma,S}$ over $S$ on locally free sheaves of finite rank.

\begin{definition}
\label{definition: graded fiber functor}
Let $S$ be a scheme over $k$. A $\Gamma$-graded fiber functor of $\mathcal{T}$ over $S$ is an exact tensor functor $$\gamma\colon \mathcal{T}\to \mathrm{Gr}^\Gamma\Bun_S.$$ 
\end{definition}

Equivalently, a $\Gamma$-graded fiber functor on $\mathcal{T}$ over $S$ consists of a usual fiber functor 
$$
\omega\colon \mathcal{T}\to \mathrm{Bun}(S),
$$ 
i.e., $\omega$ is an exact tensor functor, together with a homomorphism $D_{\Gamma,S}\to \Aut^\otimes(\omega)$ of group schemes over $S$ where $\Aut^\otimes(\omega)$ denotes the group scheme of tensor automorphisms of $\omega$ (cf.\ \cite[Chapitre IV.1]{saavedra_rivano_categories_tannakiennes}).
If $S$ is non-empty, then each $\Gamma$-graded fiber functor is automatically faithful (cf.\ \cite[Corollarie 2.10.ii)]{deligne_categories_tannakiennes}).

Now we define $\Gamma$-filtered fiber functors.
We start by defining the category $\mathrm{Fil}^\Gamma\Bun_{S}$ of $\Gamma$-filtered vector bundles on a $k$-scheme $S$.

\begin{definition}
\label{definition: filtered vector bundles}
Let $S$ be a scheme over $k$. A $\Gamma$-filtered vector bundle $\mathcal{E}$ on $S$ is a vector bundle $\mathcal{E}$ on $S$ together with subbundles, i.e., locally direct summands, 
$$
F^\lambda\mathcal{E}\subseteq \mathcal{E}
$$ 
for $\lambda\in \Gamma$ such that 
$$
F^\lambda\mathcal{E}\subseteq F^{\lambda^\prime}\mathcal{E}
$$ 
if $\lambda^\prime\leq \lambda$ and $F^{\lambda}\mathcal{E}=0$ for $\lambda\gg 0$ resp.\ $F^{\lambda}\mathcal{E}=\mathcal{E}$ for $\lambda\ll 0$. Morphisms of $\Gamma$-filtered vector bundles are morphisms 
$$
f\colon \mathcal{E}\to \mathcal{E}^\prime
$$ of vector bundles respecting the subsheaves $F^\lambda\mathcal{E}$, i.e.,\ 
$$
f(F^\lambda\mathcal{E})\subseteq F^\lambda\mathcal{E}^\prime.
$$
The category of $\Gamma$-filtered vector bundles on $S$ is denoted by $\mathrm{Fil}^\Gamma\Bun_{S}$. 
\end{definition}

If $(\mathcal{E},F^\bullet)$ is a $\Gamma$-filtered bundle we denote by $F^{>\lambda}\mathcal{E}$ for $\lambda\in \Gamma$ the vector bundle 
$$
F^{>\lambda}\mathcal{E}:=\sum\limits_{\lambda^\prime >\lambda} F^{\lambda^\prime}\mathcal{E}
$$ of $\mathcal{E}$.
We moreover denote by 
$$
\mathrm{gr}(\mathcal{E}):=\bigoplus_{\lambda\in \Gamma} \mathrm{gr}^\lambda(\mathcal{E})
$$ 
where 
$$
\mathrm{gr}^\lambda( \mathcal{E}):={F^\lambda\mathcal{E}}/{F^{>\lambda}\mathcal{E}}
$$
the associated $\Gamma$-graded vector bundle on $S$. In this way we obtain a functor
$$
\mathrm{gr}\colon \mathrm{Fil}^\Gamma\Bun_{S}\to \mathrm{Gr}^\Gamma\Bun_S.
$$
We note that for a $\Gamma$-filtered vector bundle there are, locally on $S$, only finitely many $\lambda\in \Gamma$ such that 
$$
\mathrm{gr}^\lambda(\mathcal{E})\neq 0.
$$  
If $(\mathcal{E},F^\bullet)$ and $(\mathcal{F},F^\bullet)$ are two $\Gamma$-filtered vector bundles on $S$ then we define their tensor product
$$
(\mathcal{E}\otimes_{\mathcal{O}_S}\mathcal{F},F^\bullet)
$$
by setting 
$$
F^\lambda(\mathcal{E}\otimes_{\mathcal{O}_S}\mathcal{F}):=\sum\limits_{\lambda^\prime+\lambda^{\prime\prime}=\lambda}F^{\lambda^\prime}\mathcal{E}\otimes_{\mathcal{O}_S}F^{\lambda^{\prime\prime}}\mathcal{F}.
$$
Moreover, we make $\mathrm{Fil}^\Gamma\Bun_{S}$ into an exact category (in the sense of Quillen) by requiring that a sequence
$$
0\to \mathcal{E}\to \mathcal{E}^\prime\to \mathcal{E}^{\prime\prime}\to 0
$$ 
is exact if and only if the associated sequence
$$
0\to \mathrm{gr}(\mathcal{E})\to \mathrm{gr}(\mathcal{E}^\prime)\to \mathrm{gr}(\mathcal{E}^{\prime\prime})\to 0
$$
is exact. In particular, the functor 
$$
\mathrm{gr}\colon \mathrm{Fil}^\Gamma\Bun_{S}\to \mathrm{Gr}^\Gamma\Bun_S
$$ 
is then an exact functor. Moreover, $\mathrm{gr}$ is even a tensor functor.
We can also go into the other direction and associate to a $\Gamma$-graded vector bundle a filtered one. Namely, for a $\Gamma$-graded vector bundle 
$$
\mathcal{M}=\bigoplus\limits_{\lambda\in \Gamma}\mathcal{M}^\lambda
$$ 
we can define the $\Gamma$-filtered vector bundle 
$$
(\mathrm{fil}(\mathcal{M}),F^\bullet)
$$ 
by setting 
$$
F^\lambda\mathrm{fil}(\mathcal{M}):=\bigoplus\limits_{\lambda^\prime\geq \lambda}\mathcal{M}^{\lambda^\prime}.
$$ 
In this way we obtain an exact tensor functor 
$$
\mathrm{fil}\colon \mathrm{Gr}^\Gamma\Bun_S\to \mathrm{Fil}^\Gamma\Bun_{S}.
$$

We are now able to define what a $\Gamma$-filtered fiber functor on $\mathcal{T}$ is (cf.\ \cite[Section 4.1]{ziegler_graded_and_filtered_fiber_functors}).

\begin{definition}
\label{definition: filtered fiber functor}
Let $S$ be a scheme over $k$. 
\newline
i) A $\Gamma$-filtered fiber functor of $\mathcal{T}$ over $S$ is an exact tensor functor $$\omega\colon \mathcal{T}\to \mathrm{Fil}^\Gamma\Bun_{S}.$$
ii)
A splitting of a $\Gamma$-filtered fiber functor $\omega\colon \mathcal{T}\to \mathrm{Fil}^\Gamma\Bun_{S}$ is a $\Gamma$-graded fiber functor 
$$
\gamma\colon \mathcal{T}\to \mathrm{Gr}^\Gamma\Bun_S
$$ 
such that $\omega=\mathrm{fil}\circ \gamma$. 
\end{definition}

We remark that if a $\Gamma$-filtered fiber functor $\omega$ admits a splitting then also $\Gamma$-filtered fiber functors isomorphic to $\omega$ do. Moreover, if $\omega$ admits a splitting $\gamma$, then $\gamma\cong \mathrm{gr}\circ \omega$.

Clearly, for every morphism $f\colon S^\prime\to S$ of schemes over $k$ the pullback 
$$
f^\ast\colon \Bun_{S}\to \Bun_{S^\prime}
$$ 
of vector bundles induces exact tensor functors 
$$
f^\ast\colon \mathrm{Fil}^\Gamma\Bun_{S}\to\mathrm{Fil}^\Gamma\Bun_{S^\prime}
$$ 
resp.\ 
$$
f^\ast\colon \mathrm{Gr}^\Gamma\Bun_S\to\mathrm{Gr}^\Gamma\Bun_{S^\prime}.$$ 
Moreover, 
$$
\mathrm{fil}_{S^\prime}\circ f^\ast\cong f^\ast\circ \mathrm{fil}_S\colon \mathrm{Gr}^\Gamma\Bun_S\to \mathrm{Fil}^\Gamma\Bun_{S^\prime}
$$ 
and analogous $\mathrm{gr}_{S^\prime}\circ f^\ast\cong f^\ast\circ \mathrm{gr}_S$.
Let 
$$
\mathrm{\underline{\Hom}}^\otimes(\mathcal{T},\mathrm{Fil}^\Gamma\Bun)
$$ 
resp.\ 
$$
\mathrm{\underline{\Hom}}^\otimes(\mathcal{T},\mathrm{Gr}^\Gamma\Bun)
$$ 
be the fibered categories of $\Gamma$-filtered resp.\ $\Gamma$-graded fiber functors for $\mathcal{T}$. As in \cite[Lemma 4.32, Lemma 3.12.]{ziegler_graded_and_filtered_fiber_functors} it follows that both fibered categories are stacks for the fpqc-topology. Moreover, $\mathrm{fil}$ resp.\ $\mathrm{gr}$ define morphisms
$$
\mathrm{fil}\colon \mathrm{\underline{\Hom}}^\otimes(\mathcal{T},\mathrm{Gr}^\Gamma\Bun)\to \mathrm{\underline{\Hom}}^\otimes(\mathcal{T},\mathrm{Fil}^\Gamma\Bun)
$$
resp.\
$$
\mathrm{gr}\colon \mathrm{\underline{\Hom}}^\otimes(\mathcal{T},\mathrm{Fil}^\Gamma\Bun)\to \mathrm{\underline{\Hom}}^\otimes(\mathcal{T},\mathrm{Gr}^\Gamma\Bun)
$$
of stacks.
Finally, for a $\Gamma$-filtered fiber functor 
$$
\omega\colon \mathcal{T}\to \mathrm{Fil}^\Gamma\Bun_{S}
$$ 
we denote by 
$$
\mathrm{Spl}(\omega)
$$
the functor on $S$-schemes sending $f\colon S^\prime \to S$ to the set of splittings of the $\Gamma$-filtered fiber functor $f^\ast\circ \omega$.
The functor $\mathrm{Spl}(\omega)$ is a sheaf for the fpqc-topology.  

The main result of \cite[Main theorem 4.14]{ziegler_graded_and_filtered_fiber_functors} is the following (in the case we are interested in, i.e., the band of $\mathcal{T}$ is reductive, also cf.\ \cite{cornut_filtrations_and_buildings}).

\begin{theorem}
\label{theorem: splittings of z-filtered fiber functors}
Every $\Z$-filtered fiber functor $\omega\colon \mathcal{T}\to \mathrm{Fil}^\Z\Bun_{S}$ is splittable, i.e.\ admits a splitting, fpqc-locally on $S$.    
\end{theorem}

We want to extend \Cref{theorem: splittings of z-filtered fiber functors} to more general groups $\Gamma$. As in \cite[Definition 4.8.]{ziegler_graded_and_filtered_fiber_functors} we make the following definitions.

\begin{definition}
\label{definition: group sheaves associated to filtered fiber functors}
Let 
$$
\omega\colon \mathcal{T}\to \mathrm{Fil}^\Gamma\Bun_{S}
$$ 
be a $\Gamma$-filtered fiber functor on the Tannakian category $\mathcal{T}$. We define the following group sheaves over $S$.
\newline
i)\ \ $P(\omega):=\Aut^\otimes(\omega)$
\newline
ii)\ $L(\omega):=\Aut^\otimes(\mathrm{gr}\circ \omega)$
\newline
iii) $U(\omega):=\mathrm{Ker}(P(\omega)\to L(\omega))$
\end{definition}

Moreover, we define a canonical filtration 
$$
U_\lambda(\omega)\subseteq P(\omega),\ 0\leq\lambda\in \Gamma,
$$ 
as follows. For $\lambda\in \Gamma, \lambda\geq 0$, we set $U_{\lambda}(\omega)$ as the subgroup of $P(\omega)$ consisting of elements $g\in P(\omega)$ which act trivially on 
$$
F^{\lambda^\prime}\omega(X)/F^{\lambda^\prime+\lambda}\omega(X)
$$ 
for all $\lambda^\prime \in \Gamma$ and $X\in \mathcal{T}$.
 
Clearly, for $\lambda\geq \lambda^\prime$ the group 
$$
U_{\lambda^\prime}(\omega)\subseteq U_{\lambda}(\omega)
$$ 
is normal. We denote by 
$$
U_{>\lambda}(\omega)
$$ 
the union in $U_{\lambda}(\omega)$ of all subgroups $U_{\lambda^\prime}(\omega)$ for $\lambda^\prime>\lambda$ and by 
$$
\mathrm{gr}^\lambda(U(\omega))
$$ 
the quotient 
$$
\mathrm{gr}^\lambda(U(\omega)):=U_{\lambda}(\omega)/U_{>\lambda}(\omega).
$$

We will use the following argument to deduce results for $\Gamma$-filtered fiber functors where $\Gamma\subseteq \Q$ is a subgroup from results about $\Z$-filtered fiber functors.
Let $\Gamma\subseteq \Q$ be a subgroup and let $\omega\colon \mathcal{T}\to \mathrm{Fil}^\Gamma\Bun_{S}$ be a $\Gamma$-filtered fiber functor.
Define
$$
\Gamma_\omega:=\{ \lambda\in \Gamma\ |\ \mathrm{gr}^\lambda \omega(X)\neq 0 \textrm{ for some } X\in \mathcal{T}\}.
$$
Then $\Gamma_\omega\subseteq \Gamma$ is a subgroup as $\mathrm{gr}$ and $\omega$ are tensor functors. If $\mathcal{T}$ has a tensor generator, then (by exactness of $\mathrm{gr}$ and $\omega$) the group $\Gamma_\omega\subseteq \Gamma$ is finitely generated by the $\lambda\in \Gamma$ such that $\mathrm{gr}^\lambda(\omega(X))\neq 0$ for a tensor generator $X\in \mathcal{T}$. Hence, $\Gamma_\omega$ is isomorphic to $\Z$ or $\{0\}$ in this case because $\Gamma_\omega\subseteq \Q$. Moreover, there are fully faithful embeddings
$$
\mathrm{Fil}^{\Gamma_\omega}\Bun_S\to \mathrm{Fil}^\Gamma\Bun_S
$$
resp.\
$$
\mathrm{Gr}^{\Gamma_\omega}\Bun_S\to \mathrm{Gr}^\Gamma\Bun_S
$$
and $\omega$ factors through a $\Gamma_\omega$-filtered fiber functor 
$$
\omega^\prime\colon \mathcal{T}\to \mathrm{Fil}^{\Gamma_\omega}\Bun_S.
$$ 
Moveover, 
$$
P(\omega)\cong P(\omega^\prime), L(\omega)\cong L(\omega^\prime),U(\omega)\cong U(\omega^\prime)
$$ 
etc. Thus all data for $\omega$ is defined by a $\Gamma_\omega$-filtered fiber functor to which we can apply the known results. In particular, we can conclude by \Cref{theorem: splittings of z-filtered fiber functors} that if $\mathcal{T}$ has a tensor generator, then every $\Gamma$-filtered fiber functor admits a splitting, fpqc-locally on $S$.
We record the following theorem collecting results about $\Gamma$-filtered fiber functors.

To state it let $\pi(\mathcal{T})$ be the fundamental group of $\mathcal{T}$ (cf.\ \cite[Definition 8.13]{deligne_categories_tannakiennes}), an affine group scheme in the Tannakian category $\mathcal{T}$ represented by a Hopf algebra $\mathcal{O}_{\pi(\mathcal{T})}$ in Ind-$\mathcal{T}$. For every fiber functor $\omega\colon \mathcal{T}\to \Bun_{S}$ it has the property
$$
\omega(\pi(\mathcal{T}))\cong \Aut^\otimes(\omega)
$$
as affine group schemes over $S$. More precisely, the $\mathcal{O}_S$-Hopf algebra representing $\mathrm{Aut}^\otimes(\omega)$ is isomorphic to the Hopf algebra $\omega(\mathcal{O}_{\pi(\mathcal{T}})$.
If $\mathcal{T}$ admits a tensor generator we define the Lie algebra 
$$
\mathrm{Lie}(\pi(\mathcal{T})):=(I/I^2)^\vee,
$$ 
where $I\subseteq \mathcal{O}_{\pi(\mathcal{T})}$ is the augmentation ideal of $\mathcal{O}_{\pi(\mathcal{T})}$.
If $\mathcal{T}=\Rep_k(G)$ for some affine group scheme $G$ over $k$, then 
$$
\mathcal{O}_{\pi(\mathcal{T})}=\mathcal{O}_G
$$ 
with $G$ acting on $\mathcal{O}_G$ by conjugation and we see that we recover the usual notion of the Lie algebra of $G$ with its adjoint action.

\begin{theorem}
\label{theorem: results about q-filtered fiber functors}
Let $\Gamma\subseteq \Q$ be a subgroup.
Let $\mathcal{T}$ be a Tannakian category over $k$ and let 
$$
\omega\colon \mathcal{T}\to \mathrm{Fil}^\Gamma\Bun_{S}
$$
be a $\Gamma$-filtered fiber functor.
Define 
$$
G:=\Aut^{\otimes}(\mathrm{forg}\circ \omega)
$$ 
as the group scheme over $S$ defined by the usual fiber functor
$$
\mathrm{forg}\circ \omega\colon \mathcal{T}\to \mathrm{Fil}^\Gamma\Bun_{S}\to \Bun_{S}
$$
of $\mathcal{T}$ over $S$.
Let $\pi(\mathcal{T})$ be the fundamental group of $\mathcal{T}$.
Then
\newline
i) For $\lambda\in \Gamma$, $\lambda\geq 0$, the group sheaf $U_{\lambda}(\omega)$ is representable by a group scheme, affine and faithfully flat over $S$.
\newline
ii) The affine group schemes $U(\omega)$ and $U_{\lambda}(\omega)$ for $\lambda>0$ are pro-unipotent and for $\lambda>0$ the group $\mathrm{gr}^\lambda(U(\omega))$ is abelian and pro-unipotent.
\newline
iii) $P(\omega)/U(\omega)\cong L(\omega)$
\newline
iv) If $\mathcal{T}$ admits a tensor generator, then for $\lambda\geq 0$ 
$$
\mathrm{Lie}(U_{\lambda}(\omega))\cong F^\lambda(\omega(\mathrm{Lie}(\pi(\mathcal{T}))))
$$
and
$$
\mathrm{Lie}(\mathrm{gr}^\lambda(U(\omega)))\cong \mathrm{gr}^\lambda(\omega(\mathrm{Lie}(\pi(\mathcal{T}))))
$$
where $\mathrm{Lie}(\pi(\mathcal{T}))$ is the Lie algebra of the fundamental group $\pi(\mathcal{T})$ of $\mathcal{T}$.
\newline
v) If $G$ is of finite presentation over $S$ (or equivalently, if $\mathcal{T}$ has a tensor generator), then 
$$
P(\omega),L(\omega), U(\omega), U_{\lambda}(\omega), \mathrm{gr}^\lambda(U(\omega))
$$ 
for $\lambda\geq 0$ are of finite presentation.
\newline
vi) If $G$ is smooth over $S$, then $P(\omega)$, $U(\omega)$, $U_{\lambda}(\omega)$ and $\mathrm{gr}^\lambda(U(\omega))$ for $\lambda\geq 0$ are smooth.
\newline
vii) If $G$ is reductive over $S$, then $P(\omega)$ is a parabolic subgroup scheme of $G$ and $U(\omega)$ is its unipotent radical. For $\lambda>0$ the groups 
$$
\mathrm{gr}^\lambda(U(\omega))
$$ 
are vector bundles, isomorphic to 
$$
\mathrm{Lie}(\mathrm{gr}^\lambda(U(\omega))).
$$
\newline
viii) The sheaf $\mathrm{Spl}(\omega)$ of splittings of $\omega$ is an $U(\omega)$-torsor over $S$ with respect to the fpqc-topology, in particular, $\omega$ admits a splitting fpqc-locally on $S$ and $\mathrm{Spl}(\omega)$ is represented by a scheme affine and faithfully flat over $S$.
\newline
ix) If 
$$
\gamma:=\mathrm{gr}\circ \omega\colon \mathcal{T}\to \mathrm{Gr}^\Gamma\Bun_S
$$ 
is the associated $\Gamma$-graded fiber functor of $\omega$, then for $\lambda\geq 0$
$$
U_\lambda(\omega)\cong U_\lambda(\mathrm{fil}\circ \gamma).
$$   
\end{theorem}
\begin{proof}
i) If $\mathcal{T}$ admits a tensor generator this follows from \cite[Lemma 4.20]{ziegler_graded_and_filtered_fiber_functors} and the discussion preceeding this theorem. In the general case, writing $\mathcal{T}$ as a union of sub-Tannakian categories admitting a tensor generator, shows that $U_{\lambda}(\omega)$ is an inverse limit of schemes, affine and faithfully flat over $S$ and hence itself affine and faithfully flat over $S$.
\newline
ii) The pro-unipotence of $U(\omega)$ and $U_{\lambda}(\omega)$ for $\lambda>0$ follows after taking the limit from the case that $\mathcal{T}$ has a tensor generator, say $X\in \mathcal{T}$. Then $U(\omega)$ resp.\ $U_{\lambda}(\omega)$ embeds into $\mathrm{GL}(\omega(X))$ as a subgroup of the upper triangular matrices, showing the unipotence. As in \cite[Lemma 4.21]{ziegler_graded_and_filtered_fiber_functors} it follows that $\mathrm{gr}^\lambda(U(\omega))$ is abelian. As $\mathrm{gr}^\lambda(U(\omega))$ is a quotient of $U_{\lambda}(\omega)$ it is also pro-unipotent.
\newline
iii) This follows as in \cite[Lemma 4.23]{ziegler_graded_and_filtered_fiber_functors} from viii).
\newline
iv) This follows from \cite[Proposition IV.2.1.4.1]{saavedra_rivano_categories_tannakiennes}.
\newline
v) This follows from \cite[Proposition IV.2.1.4.1]{saavedra_rivano_categories_tannakiennes}. 
\newline
vi) This follows from \cite[Lemma 4.20]{ziegler_graded_and_filtered_fiber_functors}.
\newline
vii) This follows from \cite[Lemma 4.40]{ziegler_graded_and_filtered_fiber_functors} and \cite[Proposition 4.25]{ziegler_graded_and_filtered_fiber_functors}.
\newline
viii) By \Cref{theorem: splittings of z-filtered fiber functors}, and the discussion preceeding this theorem, the statement is known if $\mathcal{T}$ admits a tensor generator. Moreover, the sheaf $\mathrm{Spl}(\omega)$ is represented by a scheme, affine and faithfully flat over $S$ by \cite[Lemma 4.20]{ziegler_graded_and_filtered_fiber_functors}. The general case follows by the usual limit argument. 
\newline
ix) For $\lambda\geq 0$ there exists a canonical map
$$
U_{\lambda}(\omega)\to U_{\lambda}(\mathrm{fil}\circ \gamma)
$$
and thus the statement that this is an isomorphism is fpqc-local on $S$ and we may assume that 
$$
\omega\cong \mathrm{fil}\circ \gamma
$$ 
is split in which case the statement is trivial.
\end{proof}

\section{Admissible fiber functors over the Fargues-Fontaine curve}
\label{section: admissible fiber functors}

Fix a local field $E$ with residue field $\F_q$, i.e., either $E\cong \F_q((t))$ or $E$ is a finite extension of $\Q_p$. Moreover, let $F/\bar{\F}_q$ be a perfectoid algebraically closed extension of an algebraic closure $\bar{\F}_q$ of $\F_q$. Let $X_{E,F}$ be the schematic Fargues-Fontaine curve associated with $E$ and $F$ (cf.\ \cite[Chapter 1]{fargues_geometrization_of_the_local_langlands_correspondence_overview}).
In the notation we will omit the field $F$ and we will just write $X_E$ for $X_{E,F}$.
Our results will not depend on the field $F$ (assuming that it is algebraically closed).

Let $\mathcal{T}$ be a Tannakian category over $E$. The basic example will be the category of representations $\Rep_E(G)$ of an affine group scheme $G$ over $E$. We record the following lemma from the Tannakian formalism (cf.\ \cite{deligne_categories_tannakiennes}).

\begin{lemma}
\label{lemma: tannakian description of torsors}
Let $k$ be a field and let $G/k$ be an affine group scheme. Then for every scheme $S$ over $k$ the groupoid of fiber functors
$$
\omega\colon \Rep_k(G)\to \Bun_{S}
$$ 
is equivalent to the groupoid of $G$-torsors over $S$ for the fpqc-topology. If $G$ is locally of finite presentation, the same holds with ``fpqc'' replaced by ``fppf''. If $G$ is locally smooth, the same holds with ``fpqc'' replaced by ``\'etale''.
\end{lemma}

The following lemma is very important for our proof of the classification of the later defined admissible fiber functors 
$$
\omega\colon \mathcal{T}\to \mathrm{Bun}_{X_E}
$$
of $\mathcal{T}$ over $X_E$.
First we recall the Harder-Narasimhan filtration for a vector bundle $\mathcal{E}$ on $X_E$ (cf.\ \cite[Section 8.2.4.]{fargues_fontaine_courbe_et_fibres_vectories_en_theorie_de_hodge_p_adique} resp.\ \cite[Corollary 11.7.]{hartl_pink_vector_bundles_with_frobenius_structure}). Namely, there exists a canonical $\Q$-filtration $\mathrm{HN}^\bullet(\mathcal{E})$ on $\mathcal{E}\in \mathrm{Bun}_{X_E}$ by subbundles 
$$
\mathrm{HN}^\lambda(\mathcal{E}), \lambda\in \Q,
$$
such that for $\lambda\in \Q$ the graded piece 
$$
\mathrm{gr}^\lambda(\mathcal{E})=\mathrm{HN}^\lambda(\mathcal{E})/{\mathrm{HN}^{>\lambda}(\mathrm{E}})
$$ 
is semistable of slope $\lambda$.

\begin{lemma}
\label{lemma: harder-narasimhan tensor functor}
Sending a vector bundle $\mathcal{E}$ to the $\Q$-filtered vector bundle 
$$
(\mathcal{E},\mathrm{HN}(\mathcal{E}))
$$
 defines a fully faithful tensor functor
$$
\mathrm{HN}\colon \mathrm{Bun}_{X_E}\to \mathrm{Fil}^\Q\Bun_{X_E}.
$$
\end{lemma}
\begin{proof}
The Harder-Narasimhan filtration is preserved by every morphism of vector bundles. Hence the functor $\mathrm{HN}$ is well-defined and fully faithful.
The classification of vector bundles on $X_E$ (cf.\ \cite[Chapter 8]{fargues_fontaine_courbe_et_fibres_vectories_en_theorie_de_hodge_p_adique}, \cite{hartl_pink_vector_bundles_with_frobenius_structure} resp.\ \Cref{theorem: classification of vector bundles}) implies that the tensor product of two semistable vector bundles of slope $\lambda$ resp.\ $\mu$ is again semistable of slope $\lambda+\mu$. This implies that $\mathrm{HN}$ is moreover a tensor functor. 
\end{proof}

However, note that the functor $\mathrm{HN}$ is \textit{not} exact, because semistable vector bundles are successive extensions of line bundles which are possibly of different slopes.
Therefore we make the following definition.\footnote{In \cite{saavedra_rivano_categories_tannakiennes} Saavedra-Rivano calls filtered fiber functors admissible if they are fpqc-locally splittable. By \cite{ziegler_graded_and_filtered_fiber_functors} this notion is obsolete and thus we think that our terminology is not very confusing. Also our admissible fiber functors are not equipped with a filtration as would be the case in Saavedra-Rivano's notation.}

\begin{definition}
\label{definition: admissible fiber functor}
A fiber functor $\omega\colon \mathcal{T}\to \mathrm{Bun}_{X_E}$ is called admissible if the composition
$$
\mathcal{T}\xrightarrow{\omega}\mathrm{Bun}_{X_E}\xrightarrow{\mathrm{HN}}\mathrm{Fil}^\Q\Bun_{X_E}
$$ 
is again exact, i.e., a $\Q$-filtered fiber functor.
\end{definition}

For example, if $\mathrm{char}(E)=0$ and $\mathcal{T}\cong \Rep_E(G)$ for a reductive group $G$ over $E$, then every fiber functor $\omega\colon \mathcal{T}\to \mathrm{Bun}_{X_E}$ is admissible as $\mathcal{T}$ is semisimple in this case.

Admissibility is a convenient notion as it can be checked after base change along a finite field extension $E^\prime/E$.

\begin{lemma}
\label{lemma: admissible fiber functors and finite base change}
Let $E^\prime/E$ be a finite field extension and let $\omega\colon \mathcal{T}\to \mathrm{Bun}_{X_E}$ be a fiber functor. Then $\omega$ is admissible if and only if the composition 
$$
\mathcal{T}\xrightarrow{\omega} \mathrm{Bun}_{X_E}\xrightarrow{f^\ast} \mathrm{Bun}_{X_{E^\prime}}
$$ 
is admissible where $f\colon X_{E^\prime}\cong X_E\otimes_E E^\prime\to X_E$ is the canonical morphism.   
\end{lemma}
\begin{proof}
Because $f$ is faithfully flat it suffices to show that the diagram
$$
\xymatrix{
\mathrm{Bun}_{X_{E}}\ar[r]^-{\mathrm{HN}}\ar[d]^{f^\ast} & \mathrm{Fil}^\Q\Bun_{X_E}\ar[d]^{f^\ast}\\
\mathrm{Bun}_{X_{E^\prime}}\ar[r]^-{\mathrm{HN}} & \mathrm{Fil}^\Q\Bun_{X_{E^\prime}}
}
$$
commutes, i.e., if $\mathcal{E}$ is a vector bundle on $X_E$, then the pullback of the Harder-Narasimhan filtration of $\mathcal{E}$ on $X_E$ is the Harder-Narasimhan filtration of $f^\ast\mathcal{E}$ on $X_{E^\prime}$.
We assume first that $E^\prime$ over $E$ is separable.
 Let $\mathcal{E}$ be a vector bundle on $X_E$. If $\mu$ is the slope of $\mathcal{E}$, then $[E^\prime:E]\mu$ is the slope of $f^\ast\mathcal{E}$ because $f^\ast$ preserves ranks, but multiplies degrees by $[E^\prime:E]$. In particular, if $f^\ast\mathcal{E}$ is semistable, then $\mathcal{E}$ is semistable, because for a destabilizing subbundle $\mathcal{F}$ of $\mathcal{E}$ the pullback $f^\ast\mathcal{F}$ would again be destabilizing.
This implies that it suffices to prove the claim in the case that $E^\prime/E$ is Galois. Then the Harder-Narasimhan filtration of $f^\ast\mathcal{E}$ must be stable under the Galois action and hence descends to $X_E$. Moreover, this descended filtration must be the Harder-Narasimhan filtration of $\mathcal{E}$ as its graded pieces are semistable of strictly decreasing slopes because this holds after pullback to $X_{E^\prime}$. In particular, the pullback of the Harder-Narasimhan filtration of $\mathcal{E}$ is the Harder-Narasimhan filtration of $f^\ast\mathcal{E}$. 

Now assume that $E^\prime/E$ is purely inseparable. Again it is clear that a vector bundle $\mathcal{E}\in \Bun_{X_E}$ is semistable if $f^\ast\mathcal{E}$ is semistable. Conversely, let $\mathcal{E}\in \Bun_{X_E}$ be a semistable vector bundle. We want to show that $f^\ast\mathcal{E}$ is semistable. For this we may assume that $\mathcal{E}$ is simple. In this case there exists $d\in \Z$ and $d\geq 0$, s.t.
$$
\mathcal{E}\cong \pi_\ast\mathcal{O}_{X_{E_h}}(d)
$$ 
where $E_h$ is an unramified extension of $E$ of degree $h$ and
$$
\pi\colon X_{E_h}\to X_E
$$  
is the canonical morphism. As $E^\prime/E$ is purely inseparable, $E^\prime\otimes_E E_h$ is isomorphic to the unramified extension $E^\prime_h$ of $E^\prime$ of degree $h$. In particular, 
$$
f^\ast\pi_\ast(\mathcal{O}_{X_{E_h}}(d))\cong \pi^\prime_\ast(\mathcal{O}_{X_{E^\prime_h}}(hd)),
$$
where $\pi^\prime\colon X_{E^\prime_h}\to X_{E^\prime}$ is the projection,
is again semistable.
\end{proof}

Let 
$$
L:=\widehat{E^\mathrm{un}}
$$ 
be the completion of the maximal unramified extension of $E$ and let 
$$
\varphi-\mathrm{Mod}_L
$$ 
be the category of isocrystals over $L$. As we have fixed an embedding $\bar{\F}_q\subseteq F$ we obtain a canonical exact tensor functor
$$
\mathcal{E}(-)\colon \varphi-\mathrm{Mod}_L\to \Bun_{X_E}
$$
(cf.\ \cite[Section 8.2.3.]{fargues_fontaine_courbe_et_fibres_vectories_en_theorie_de_hodge_p_adique} if $E/\Q_p$ resp.\ \cite[Chapter 8]{hartl_pink_vector_bundles_with_frobenius_structure} if $E=\F_q((t))$).

We recall properties of this functor.

\begin{theorem}
\label{theorem: classification of vector bundles}
The functor 
$$
\mathcal{E}(-)\colon \varphi-\mathrm{Mod}_L\to \Bun_{X_E}  
$$
is an exact faithful tensor functor inducing a bijection on isomorphism classes.
Moreover, for $\lambda\in \Q$ it induces an equivalence between semistable isocrystals of slope $-\lambda$ with semistable vector bundles of slope $\lambda$.
\end{theorem}
\begin{proof}
Cf.\ \cite[Section 8.2.4.]{fargues_fontaine_courbe_et_fibres_vectories_en_theorie_de_hodge_p_adique} resp. \cite[Proposition 8.6., Theorem 11.1.]{hartl_pink_vector_bundles_with_frobenius_structure}.  
\end{proof}

One has to be a bit careful when comparing the slope of some isocrystal 
$$
D\in \varphi-\mathrm{Mod}_L
$$ 
with the slope of the vector bundle $\mathcal{E}(D)$. If $\pi\in E$ is a uniformizer then for $n\in \Z$ the isocrystal 
$$
D:=(L,\pi^n\varphi_L)
$$
is of slope $n$ and sent to the line bundle $\mathcal{O}_{X_E}(-n)$ and not to $\mathcal{O}_{X_E}(n)$. This explains the appearance of the sign in \Cref{theorem: classification of vector bundles}.

\Cref{theorem: classification of vector bundles} implies that the functor
$$
\mathrm{gr}\circ\mathrm{HN}\colon \Bun_{X_E}\to \mathrm{Gr}^\Q\Bun_{X_E}
$$
preserves duals and symmetric resp.\ exterior powers because this is true for the functor sending an isocrystal to its decomposition into isoclinic components.

Moreover, the category $\varphi-\mathrm{Mod}_L$ is canonically $\Q$-graded by decomposing an isocrystal into its isoclinic components.
We thus obtain a functor
$$
\mathcal{E}_{\mathrm{gr}}(-)\colon \varphi-\mathrm{Mod}_L\to \mathrm{Gr}^\Q\Bun_{X_E}.
$$

\begin{lemma}
\label{lemma: isocrystals as graded bundles}
The functor $\mathcal{E}_{\mathrm{gr}}(-)$ identifies the category $\varphi-\mathrm{Mod}_L$ with the full subcategory of $\mathrm{Gr}^\Q\Bun_{X_E}$ consisting of $\Q$-graded vector bundles 
$$
\mathcal{E}=\bigoplus\limits_{\lambda\in \Q}\mathcal{E}^\lambda
$$ 
such that for $\lambda\in \Q$ the vector bundle $\mathcal{E}^\lambda$ is semistable of slope $\lambda$.  
\end{lemma}
\begin{proof}
  This follows from the classification of vector bundles on $X_E$ and their homomorphisms (cf.\ \Cref{theorem: classification of vector bundles}).
\end{proof}

We moreover see that the functor
$$
\varphi-\mathrm{Mod}_L\xrightarrow{\mathcal{E}(-)}\mathrm{Bun}_{X_E}\xrightarrow{\mathrm{HN}} \mathrm{Fil}^\Q\Bun_{X_E}\xrightarrow{\mathrm{gr}}\mathrm{Gr}^\Q\Bun_{X_E}
$$
is fully faithful and exact as it is isomorphic to the functor $\mathcal{E}_{\mathrm{gr}}(-)$. In the following we will sometimes just write $\mathcal{E}(-)$ instead of $\mathcal{E}_{\mathrm{gr}}(-)$.

We now start to classify admissible fiber functors.

\begin{lemma}
\label{lemma: exact tensor functors over isocrystals define admissible fiber functors}
Let $\omega^\prime\colon \mathcal{T}\to \varphi-\mathrm{Mod}_L$ be an exact tensor functor. Then the fiber functor 
$$
\mathcal{E}(-)\circ \omega^\prime\colon\mathcal{T}\to \mathrm{Bun}_{X_E}
$$ 
is admissible.   
\end{lemma}
\begin{proof}
It suffices to show that the tensor functor 
$$
\mathrm{gr}\circ \mathrm{HN}\circ \mathcal{E}(-)\circ \omega^\prime
$$ 
is exact. But as was noted above the functor 
$$
\mathcal{E}_{\mathrm{gr}}(-)\cong \mathrm{gr}\circ \mathrm{HN}\circ \mathcal{E}(-)
$$ is exact, which implies the claim.  
\end{proof}

Let $\mathcal{T}$ be a Tannakian category and let 
$$
\omega\colon \mathcal{T}\to \Bun_{S}
$$ 
be a fiber functor of $\mathcal{T}$ over a non-empty scheme $S$. If the group scheme $\Aut^\otimes(\omega)$ of tensor automorphisms is smooth over $S$, then this is true for every fiber functor of $\mathcal{T}$ as fiber functors are fpqc-locally isomorphic. Hence, this property of smoothness is in fact intrinsic to $\mathcal{T}$ (and can be rephrased by saying that the band of the gerbe associated to $\mathcal{T}$ is smooth).

We now prove our main theorem about admissible functors.

\begin{theorem}
\label{theorem: factorization of admissible fiber functors}
Let $\mathcal{T}$ be a Tannakian category admitting a tensor generator and let 
$$
\omega\colon \mathcal{T}\to \mathrm{Bun}_{X_E}
$$ 
be an admissible fiber functor such that $\Aut^\otimes(\omega)$ is represented by a reductive group scheme over $X_E$. Then $\omega$ factors as 
$$
\omega\cong \mathcal{E}(-)\circ \omega^\prime
$$ 
for an exact tensor functor 
$$\omega^\prime\colon \mathcal{T}\to \varphi-\mathrm{Mod}_L.$$
\end{theorem}
\begin{proof}
We consider the tensor functor
$$
\psi:=\mathrm{HN}\circ \omega\colon \mathcal{T}\to \mathrm{Fil}^\Q\Bun_{X_E}.
$$
By assumption this tensor functor is a $\Q$-filtered fiber functor, i.e., it is exact.
Therefore it gives rise to the $U(\psi)$-torsor 
$$\mathrm{Spl}(\psi)$$
of splittings of $\psi$ (cf.\ \Cref{theorem: results about q-filtered fiber functors}).
Moreover, the group scheme $U(\psi)$ over $X_E$ is smooth (by \Cref{theorem: results about q-filtered fiber functors}) as $\Aut^\otimes(\omega)$ is assumed to be smooth, unipotent and admits a filtration 
$$U(\psi)\supseteq U_{\lambda}(\psi)$$ 
for $\lambda\geq 0$. Moreover, for $\lambda>0$ the associated graded pieces $\mathrm{gr}^\lambda(U(\gamma))$ are vector bundles with 
$$
\mathrm{Lie}(\mathrm{gr}^\lambda(U(\psi)))\cong \mathrm{gr}^\lambda(\mathrm{HN}\circ \omega(\mathrm{Lie}(\pi(\mathcal{T})))).
$$ 
where $\pi(\mathcal{T})$ is the fundamental group of $\mathcal{T}$ (cf.\ \Cref{theorem: results about q-filtered fiber functors}). In particular, 
$$
\mathrm{gr}^\lambda(U(\psi))\cong \mathrm{Lie}(\mathrm{gr}^\lambda(U(\psi)))
$$ 
is semistable of slope $\lambda>0$. As only finitely many $\mathrm{gr}^\lambda(U(\psi))$ are non-zero we can conclude
$$
H^1_{\acute{e}t}(X_E,U(\psi))=0
$$  
as for every semistable vector bundle $\mathcal{E}$ of slope $\lambda\geq 0$ on $X_E$ the cohomology group
$$H^1_{\acute{e}t}(X_E,\mathcal{E})=0$$
vanishes (cf.\ \cite[Chapter 8]{fargues_fontaine_courbe_et_fibres_vectories_en_theorie_de_hodge_p_adique} resp.\ \cite{hartl_pink_vector_bundles_with_frobenius_structure}).
In particular, the $U(\psi)$-torsor $\mathrm{Spl}(\psi)$-torsor is trivial and we see that we can choose a splitting 
$$
\gamma\colon \mathcal{T}\to \mathrm{Gr}^\Q\Bun_{X_E}
$$ 
of $\psi$, i.e., 
$$\psi=\mathrm{fil}\circ \gamma.$$
By construction for every $V\in \mathcal{T}$ and every $\lambda\in \Q$ the locally free sheaf 
$$
\mathrm{gr}^\lambda(\gamma(V))
$$ 
must be semistable of slope $\lambda$. In other words (using \Cref{lemma: isocrystals as graded bundles}), the functor 
$$
\gamma\colon \mathcal{T}\to \mathrm{Gr}^\Q\Bun_{X_E}
$$ 
factors as 
$$
\mathcal{E}_{\mathrm{gr}}(-)\circ \omega^\prime=\gamma
$$ 
for an exact tensor functor 
$$
\omega^\prime\colon \mathcal{T}\to \varphi-\mathrm{Mod}_L.
$$
Forgetting the grading shows $\mathcal{E}(-)\circ\omega^\prime\cong \omega$, which proves the claim. 
\end{proof}

\begin{definition}
\label{definition: b of t}
Let $\mathcal{T}$ be a Tannakian category over $E$. We define 
$\mathcal{B}(\mathcal{T})$ as the groupoid of exact tensor functors
$$
\omega^\prime\colon \mathcal{T}\to \varphi-\mathrm{Mod}_L
$$  
and $B(\mathcal{T})$ as the set of isomorphism classes of such.
\end{definition}

By \cite[Lemma 9.1.4.]{dat_orlik_rapoport_period_domains} (cf.\ \Cref{proposition: explicit description of b of g}) this definition agrees with Kottwitz original definition of the set $B(G)$ as $\varphi$-conjugacy classes in $G(L)$ if $\mathcal{T}=\Rep_E(G)$ for a connected reductive group $G$ over $E$.

Recall that we denote by 
$$
\mathrm{\underline{\Hom}}^\otimes(\mathcal{T},\Bun_{S})
$$ 
the category of fiber functors of $\mathcal{T}$ over $S$. For $S=X_E$ we let 
$$
\mathrm{\underline{\Hom}}^\otimes_{\mathrm{adm}}(\mathcal{T},\mathrm{Bun}_{X_E})
$$
be the full subcategory of admissible fiber functors of $\mathcal{T}$ over $X_E$.
From \Cref{theorem: factorization of admissible fiber functors} we obtain the following classification of admissible fiber functors.

\begin{theorem}
\label{theorem: classification of admissible fiber functors}
The composition of the functors
$$
\mathcal{B}(\mathcal{T})\to \mathrm{\underline{\Hom}}^\otimes_{\mathrm{adm}}(\mathcal{T},\mathrm{Bun}_{X_E}),\ 
\omega^\prime\mapsto \mathcal{E}(-)\circ \omega^\prime
$$
and 
$$
\mathrm{\underline{\Hom}}^\otimes_{\mathrm{adm}}(\mathcal{T},\mathrm{Bun}_{X_E})\to \mathcal{B}(\mathcal{T}),\ 
\omega\mapsto \mathrm{gr}\circ\mathrm{HN}\circ \omega 
$$
is naturally isomorphic to the identity.
These functors induce a bijection of $B(\mathcal{T})$ with the set of isomorphism classes in $\mathrm{\underline{\Hom}}^\otimes_{\mathrm{adm}}(\mathcal{T},\mathrm{Bun}_{X_E})$. 
\end{theorem}
\begin{proof}
The first statement is clear by \Cref{lemma: isocrystals as graded bundles}.
By \Cref{theorem: factorization of admissible fiber functors} the canonical map from $B(\mathcal{T})$ to isomorphism classes of admissible fiber functors is surjective, hence bijective by the first statement.
\end{proof}

Let $V\in \Rep_E(G)$ be a representation. We will need to consider the symmetric tensors $\mathrm{TS}_r(V)$ of $V$ (\cite[Chapitre 3]{roby_lois_polynomes_et_lois_formelles_en_theorie_des_modules}), i.e., $\mathrm{TS}_r(V)\subseteq V^{\otimes r}$ is the set of invariants for the permutation action of $S_r$ on $V^{\otimes r}$. By \cite[Th\'eor\`eme IV.1., Proposition IV.5.]{roby_lois_polynomes_et_lois_formelles_en_theorie_des_modules} the vector space $\mathrm{TS}_r(V)$ has the following universal property:
For every homogenous polynomial 
$$
f\colon V\to k
$$ 
of degree $r$ there exists a unique linear form 
$$
\tilde{f}\colon \mathrm{TS}_r(V)\to k
$$ 
such that $f(v)=\tilde{f}(v\otimes\ldots\otimes v)$ for all $v\in V$.
If $f$ is $G$-invariant, then $\tilde{f}$ will be $G$-invariant as well.

From a different view point the module $\mathrm{TS}_r(V)$ is isomorphic to the $r$-the divided power $\Gamma_r(V)$ of $V$ (\cite[Proposition IV.5.]{roby_lois_polynomes_et_lois_formelles_en_theorie_des_modules}). It is clear from the definition that the module of symmetric tensors $\mathrm{TS}_r(V^\vee)$ of the dual $V^\vee$ of $V$ is canonically isomorphic (as a $G$-representation) to the dual of the $r$-th symmetric power $\mathrm{Sym}^r(V)$ of $V$.

Using these symmetric tensors we obtain the following result, generalizing the main theorem of \cite{fargues_g_torseurs_en_theorie_de_hodge_p_adique} if $E=\F_q((t))$.

\begin{theorem}
\label{theorem: fargues theorem}
Let $G$ be a reductive group over $E$. Then there is a canonical bijection
$$
B(G)\cong H^1_{\acute{e}t}(X_E,G).
$$
In other words, every fiber functor
$$
\omega\colon \Rep_E(G)\to \Bun_{X_E}
$$  
is admissible.
\end{theorem}
\begin{proof}
By \Cref{theorem: classification of admissible fiber functors} it suffices to prove the second statement that every fiber functor 
$$
\omega\colon \mathrm{Rep}_E(G)\to \mathrm{Bun}_{X_E}
$$ 
is automatically admissible because the isomorphism classes of fiber functors of $\Rep_E(G)$ over $X_E$ are in bijection with $G$-torsors (for the \'etale topology) over $X_E$ (cf.\ \Cref{lemma: tannakian description of torsors}).
Let
$$
\omega\colon \Rep_E(G)\to \mathrm{Bun}_{X_E}
$$  
be a fiber functor. We want to show that the composite
$$
\mathrm{HN}\circ \omega\colon \Rep_E(G)\to \mathrm{Fil}^\Q\Bun_{X_E}
$$
is still exact. Equivalently, we can check this after taking the associated graded, i.e., we can check that the functor
$$
\mathrm{gr}\circ \mathrm{HN}\circ \omega\colon \Rep_E(G)\to \mathrm{Gr}^\Q\Bun_{X_E}
$$
is exact. As was remarked after \Cref{theorem: classification of vector bundles} all of the functors $\mathrm{gr}\circ \mathrm{HN}$ and $\omega$ are compatible with duals and symmetric resp.\ exterior powers. Moreover, they preserve modules of symmetric tensors, i.e., $S_r$-invariants in the $r$-th tensor powers.  
Using this a slightly modified proof as in \cite[Theorem 5.3.1.]{dat_orlik_rapoport_period_domains} works. We fill out the details. 
Let
$$
0\to V\to V^\prime\to V^{\prime\prime}\to 0
$$
be an exact sequence in $\Rep_E(G)$. Set
$$
\omega^\prime:=\mathrm{gr}\circ \mathrm{HN}\circ \omega.
$$
 It suffices to prove that the morphism
$$
\omega^\prime(\alpha)\colon \omega^\prime(V)\to \omega^\prime(V^\prime)
$$
is injective. Indeed, dualizing the sequence $0\to V\to V^\prime\to V^{\prime\prime}\to 0$ shows that then the morphism
$$
\omega^\prime(V^\prime)\to \omega^\prime(V^{\prime\prime})
$$
has a torsion cokernel. But this morphism is a direct sum of morphisms between semistable vector bundles of the same slope, which implies that this morphism is already surjective. Hence the morphism
$$
\omega^\prime(V^\prime)\to \omega^\prime(V^{\prime\prime})
$$
is surjective, which then implies that the sequence
$$
0\to \omega^\prime(V)\to \omega^\prime(V^\prime)\to \omega^\prime(V^{\prime\prime})\to 0
$$
is exact, because
$$
\mathrm{rk}(\omega^\prime(V^\prime))=\mathrm{rk}(\omega^\prime(V^\prime))+\mathrm{rk}(\omega^\prime(V^{\prime\prime})).
$$
Hence, we are left with proving that the morphism
$$
\omega^\prime(\alpha)\colon \omega^\prime(V)\to \omega^\prime(V^\prime)
$$
is injective. But $\omega^\prime(\alpha)$ is injective if and only this is true after taking the exterior power $\Lambda^{\dim V}$. As the functors $\mathrm{gr}\circ \mathrm{HN}$ and $\omega$ preserve exterior powers this reduces the claim to the case that $V$ is of dimension one. Tensoring with the dual of $V$ reduces to the case that $V$ is a trivial $G$-representation.
It suffices to proof that $\omega^\prime(\alpha)$
is non-zero. By Haboush's theorem (cf.\ \cite{haboush_reductive_groups_are_geometrically_reductive}) there exists some $r\geq 0$ and a $G$-equivariant homogenous polynomial
$$
f\colon V^\prime \to k
$$ 
such that $f_{|V}\neq 0$.
Using the universal property of $\mathrm{TS}_r(V^\prime)$ we obtain a $G$-equivariant linear form
$$
\tilde{f}\colon \mathrm{TS}_r(V^\prime)\to k
$$
such that $\tilde{f}$ is non-zero when restricted to $V\cong \mathrm{TS}_r(V)$.
In particular, the morphism
$$
\mathrm{TS}_r(V)\to \mathrm{TS}_r(V^\prime) 
$$
is split. This implies, using the compatibility of $\omega^\prime$ with symmetric tensors, that the morphism
$$
\mathrm{TS}_r(\omega^\prime(V))\to \mathrm{TS}_r(\omega^\prime(V^\prime)) 
$$
is split as well.
In particular, 
$$
\omega^\prime(\alpha)\colon \omega^\prime(V)\to \omega^\prime(V^\prime)
$$
is non-zero. This finishes the proof.
\end{proof}

We want to explain why our proof differs from the one in \cite[Theorem 5.3.1]{dat_orlik_rapoport_period_domains}. In fact the proof there is wrong and we give a counter example to their argument. However, using the above argument with symmetric tensors instead of symmetric powers, the proof in \cite[Theorem 5.3.1.]{dat_orlik_rapoport_period_domains} can be fixed. 
The basic problem is that in positive characteristic the dual of a symmetric power $\mathrm{Sym}^r(V)$ is not canonically isomorphic to the $r$-th symmetric power $\mathrm{Sym}^r(V^\vee)$ of the dual $V^\vee$ but to the $r$-th module $\mathrm{TS}_r(V^\vee)$ of symmetric tensors of $V^\vee$.
As was remarked before \Cref{theorem: fargues theorem} the module $\mathrm{TS}_r(V)$ of symmetric tensors is canonically isomorphic to the $r$-th divided power $\Gamma_r(V)$ of $V$. We will rather speak about divided powers in the following.

To give the counter example we assume that $E$ is some field of characteristic $2$.
Let $G:=\mathrm{SL}_2$ and take $V:=E^2$ as the standard representation of $G$ with standard basis $x,y\in V$. Note that $V$ is a selfdual representation using the paring
$$
V\times V\to E,\ ((x_1,y_1),(x_2,y_2))\mapsto x_1y_2+y_1x_2.
$$
Let $W:=\Gamma_2(V)$ be the second divided power of $V$.
Then there is a short (non-split) exact sequence
$$
0\to V_1\to W\to \mathrm{Fr}^\ast(V)\to 0,
$$
where $V_1$ is spanned by the element $x^{[1]}y^{[1]}\in W$ and where $\mathrm{Fr}^\ast(V)$ denotes the Frobenius twist of $V$.
Note that $G$ acts trivially on $V_1$.
We claim that there does not exist some $r\geq 1$ such that the morphism $V_1\cong \mathrm{Sym}^r(V_1)\to \mathrm{Sym}^r(W)$ is split.
In fact, we prove that for every $r\geq 1$ the restriction morphism
$$
(\mathrm{Sym}^r(W)^\vee)^G\to \mathrm{Sym}^r(V_1)^\vee
$$
from $G$-invariants in $\mathrm{Sym}^r(W)^\vee$ to $\mathrm{Sym}^r(V_1)^\vee$ is zero.
Using selfduality of $V$ and the fact that symmetric and divided powers are dual to each other it suffices to show that the morphism
$$
(\mathrm{Sym}^r(W)^\vee)^G\cong (\Gamma_r(\mathrm{Sym}^2(V)))^G\to \Gamma_r(V_2)
$$
is zero where $V_2$ denotes the quotient $\mathrm{Sym}^2(V)/{\mathrm{Fr}^\ast(V)}$ (concretely 
$$
V_2=\langle x^2,y^2,xy\rangle/{\langle x^2,y^2\rangle}).
$$
Set $A:=x^2, B:=y^2, C:=xy\in \mathrm{Sym}^2(V)$. A general matrix
$$
\begin{pmatrix}
a & b \\
c & d 
\end{pmatrix}\in G=\mathrm{SL}_2
$$
maps these elements to
$$
\begin{matrix}
A & \mapsto & a^2A+c^2B \\
B & \mapsto & b^2A+d^2B \\
C & \mapsto & C+abA+cdB.
\end{matrix}
$$
It suffices to show that there does not exist a $G$-invariant element $f\in \Gamma_r(\mathrm{Sym}^2(V))$ of the form
$$
f=C^{[r]}+g
$$
with $g\in \Gamma_r(\mathrm{Sym}^2(V))$ only involving divided power monomials $A^{[k]}B^{[l]}C^{[h]}$ with $h<r$. Namely, elements like $g$ span the kernel of the restriction map to $V_2$.
Let $T\subseteq G$ be the standard torus. The element
$$
\begin{pmatrix}
\lambda & 0 \\
0 & \lambda^{-1}  
\end{pmatrix}\in T
$$
maps $A\mapsto \lambda^2A$, $B\mapsto \lambda^{-2}B$ and $C\mapsto C$.
Therefore we can conclude that every $G$-invariant element $f\in \Gamma_r(\mathrm{Sym}^r(V))$ must be of the form
$$
f=\sum\limits_{k+2l=r}a_{l,k}A^{[l]}B^{[l]}C^{[k]},
$$ 
i.e.\ in each divided power monomial occuring in $f$ the divided powers of $A$ match the divided powers of $B$.
Now apply the element 
$$
\gamma=\begin{pmatrix}
 1 & 1 \\
0 & 1  
\end{pmatrix}
$$
to $f$.
This yields
$$
f=\gamma f=\sum\limits_{k+2l=r}a_{l,k}A^{[l]}(B+A)^{[l]}(C+A)^{[k]}
$$
by $G$-invariance of $f$.
Now we collect the coefficient of $A^{[r]}$ using the rules for calculating with divided powers. Namely, this coefficient is given by
$$
\kappa:=\sum\limits_{k+2l=r}a_{l,k}((l,l,k))
$$
where $((l,l,k)):=\frac{(l+l+k)!}{l!l!k!}$.
By $T$-equivariance of $f$ we know that $\kappa=0$.
We claim moreover that
$$
\kappa=a_{0,r}
$$
which implies $a_{0,r}=0$, i.e., our claim about $f$.
It suffices to see that if $k+2l=r$ the number $$((l,l,k))$$ is divisible by $2$ if $k\neq r$ as we are over a field of characteristic $2$. 
But
$$
((l,l,k))=\binom{r}{2l}\binom{2l}{l}
$$
and $\binom{2l}{l}$ is divisible by $2$ if $l\neq 0$ by Pascal's identity
$$
\binom{2l}{l}=\binom{2l-1}{l}+\binom{2l-1}{l-1}=2\binom{2l-1}{l}.
$$

\section{Applications}
\label{section: applications}

\Cref{theorem: fargues theorem} has applications to local class field theory and to the \'etale and flat cohomology of $X_E$ with finite coefficients (cf.\ \cite[Section 3]{fargues_g_torseurs_en_theorie_de_hodge_p_adique}). In this section we present these applications to handle, similarly to \cite{fargues_g_torseurs_en_theorie_de_hodge_p_adique} for $E/\Q_p$, the case $E=\F_q((t))$. However, the methods for $E=\F_q((t))$ are very similar to the one for $E/\Q_p$ and we present the statements uniformly for both cases.

We keep the notations from the last section. Moreover, we denote by $\bar{E}$ a separable closure of $E$.
First we recall the calculation of the \'etale fundamental group $\pi_1^{\acute{e}t}(X_E,\bar{x})$ of $X_E$.

\begin{theorem}
\label{theorem: etale fundamental group of the curve}
For every geometric point $\bar{x}$ of $X_E$ the canonical morphism
$$
\pi^{\acute{e}t}_1(X_E,\bar{x})\to\Gal(\bar{E}/E)
$$  
is an isomorphism.
\end{theorem}
\begin{proof}
This is proven in \cite[Th\'eor\`eme 8.6.1.]{fargues_fontaine_courbe_et_fibres_vectories_en_theorie_de_hodge_p_adique} for $E/\Q_p$, but the same proof applies to $E=\F_q((t))$ using \Cref{theorem: classification of vector bundles}.  
\end{proof}

Let $\mathrm{Br}(X_E)$ be the Brauer group of $X_E$. Because $X_E$ is noetherian of Krull dimension $1$ \cite[Corollaire 2.2.]{grothendieck_le_group_de_brauer_II} implies  $$
\mathrm{Br}(X_E)\cong H^2_{\acute{e}t}(X_E,\mathbb{G}_m).
$$

\begin{theorem}
\label{theorem: brauer group of the curve}
We have
$$
\mathrm{Br}(X_E)=H^2_{\acute{e}t}(X_E,\mathbb{G}_m)=0.
$$  
\end{theorem}
\begin{proof}
By definition, every element in $\mathrm{Br}(X_E)$ is in the image of 
$$
H^1_{\acute{e}t}(X_E,\mathrm{PGL}_n)\to H^2_{\acute{e}t}(X_E,\mathbb{G}_m)
$$ 
for some $n\geq 0$.
By \Cref{theorem: fargues theorem} there is a commutative diagram
$$
\xymatrix{
B(\mathrm{GL}_n)\ar[r]\ar[d]^\cong & B(\mathrm{PGL}_n)\ar[d]^\cong \\
H^1_{\acute{e}t}(X_E,\mathrm{Gl}_n)\ar[r]& H^1_{\acute{e}t}(X_E,\mathrm{PGL}_n).
}
$$
Moreover, using \Cref{proposition: explicit description of b of g} the top horizontal arrow is trivially surjective. In particular, every element in $H^1_{\acute{e}t}(X_E,\mathrm{PGL}_n)$ maps to $0$ in $\mathrm{Br}(X_E)$. Thus, 
$$
\mathrm{Br}(X_E)=0.
$$
\end{proof}

Let 
$$
f\colon X_E\to \Spec(E)
$$ 
be the canonical morphism. As $E^\prime=H^0(X_{E^\prime},\mathcal{O}_{E^\prime})$ for every finite extension $E^\prime$ of $E$ we have
$$
f_\ast(\mathbb{G}_m)=\mathbb{G}_m.
$$
Using the Leray spectral sequence for $f$ and \Cref{theorem: brauer group of the curve} one obtains an exact sequence
$$
0\to H^1_{\acute{e}t}(X_E,\mathbb{G}_m)\to H^0(\Gal(\bar{E}/E),H^1_{\acute{e}t}(X_{\bar{E}}, \mathbb{G}_m))\to H^2_{\acute{e}t}(\Spec(E),\mathbb{G}_m)\to 0.
$$
But 
$$
H^1_{\acute{e}t}(X_E,\mathbb{G}_m)\cong \Z
$$ 
and 
$$
H^1_{\acute{e}t}(X_{\bar{E}},\mathbb{G}_m)\cong \Q
$$
with trivial $\Gal(\bar{E}/E)$-action.
In particular, we obtain the computation of the Brauer group of $E$
$$
\mathrm{Br}(E)\cong \Q/\Z
$$
in a rather complicated way. One should note that in order to deduce this theorem we implicitly used Steinbergs theorem that 
$$
H^1_{\acute{e}t}(\Gal(\bar{E}/{E^{\mathrm{un}}}), G)=1
$$ 
for a (connected) reductive group $G$ over the maximal unramified extension $E^{\mathrm{un}}$ of $E$ to deduce the concrete description of $B(G)$ (cf.\ \Cref{proposition: explicit description of b of g}) and therefore the surjectivity of 
$$
B(\mathrm{GL}_n)\to B(\mathrm{PGL}_n).
$$

\begin{theorem}
\label{theorem: etale cohomology of the curve}
Let $A$ be a discrete torsion module for $\Gal(\bar{E}/E)$ with $nA=0$ for some $n$ prime to the characteristic of $E$. Then for $i\leq 2$ the canonical morphism
$$
H^i(\Gal(\bar{E}/E),A)\to H^i_{\acute{e}t}(X_E,A)
$$  
is an isomorphism.
\end{theorem}
\begin{proof}
Let $f\colon X_E\to \Spec(E)$ be the canonical morphism. It suffices to prove
$$
R^if_\ast(\mu_l)=0
$$
for $i\leq 2$ and $l$ prime to $\mathrm{char}(E)$.
Because $l$ and $\mathrm{char}(E)$ are coprime the Kummer sequence
$$
0\to \mu_l\to \mathbb{G}_m\to \mathbb{G}_m\to 0
$$
is an exact sequence of \'etale sheaves. But
$$
H^0_{\acute{e}t}(X_{\bar{E}},\mathbb{G}_m)\cong \bar{E}^\times
$$
$$
H^1_{\acute{e}t}(X_{\bar{E}},\mathbb{G}_m)\cong \Q
$$
and 
$$
H^2_{\acute{e}t}(X_{\bar{E}},\mathbb{G}_m)\cong 0
$$
(cf.\ \Cref{theorem: brauer group of the curve}).
This implies the claim.
\end{proof}

Now we discuss the case $E=\F_q((t))$ with $p$-torsion coefficients using flat cohomology.

\begin{theorem}
\label{theorem: flat cohomology of the curve}  
Let $D$ be a finite multiplicative group scheme over $E$.
Then the canonical morphism
$$
H^i_{\mathrm{fl}}(\Spec(E),D)\to H^i_{\mathrm{fl}}(X_E,D)
$$
is an isomorphism for $i\leq 2$.
Moreover, if $E=\F_q((t))$ and $D=\F_p$, then
$$
H^i_{\mathrm{fl}}(\Spec(E),\F_p)\cong H^i_{\mathrm{fl}}(X_E,\F_p)
$$
for $i\geq 1$.
In particular, 
$$
H^i_{\mathrm{fl}}(X_E,\F_p)=0
$$
for $i\geq 2$.
\end{theorem}
\begin{proof}
Let again $f\colon X_E\to \Spec(E)$ be the canonical morphism. As in \Cref{theorem: etale cohomology of the curve} it suffices to prove for every prime $l\in \Z$
$$
R^if_\ast(\mu_l)=0
$$
resp.\
$$
R^if_\ast(\F_p)=0
$$
for $1\leq i\leq 2$ resp.\ all $i\geq 1$.
Let $E^{\mathrm{alg}}$ be an algebraic closure of $E$. It suffices to prove
$$
H^i_{\mathrm{fl}}(X_{E^{\mathrm{alg}}},\mu_l)=0
$$
for $1\leq i\leq 2$ resp.\ 
$$
H^i_{\mathrm{fl}}(X_{E^{\mathrm{alg}}},\F_p)=0
$$
for $i\geq 1$.
For coefficients $\mu_l$ the statement follows as in \Cref{theorem: etale cohomology of the curve} using that for $l=p$ the Kummer sequence is exact for the flat topology.
The second statement follows using the Artin-Schreier sequence
$$
0\to \F_p\to \mathcal{O}_{X_{E^{\mathrm{alg}}}}\to \mathcal{O}_{X_{E^{\mathrm{alg}}}}\to 0
$$
together with the facts that
$$
H^i_{\mathrm{fl}}(X_{E^{\mathrm{alg}}},\mathcal{O}_{X_{E^{\mathrm{alg}}}})\cong H^i_{\acute{e}t}(X_{X_{E^{\mathrm{alg}}}},\mathcal{O}_{X_{E^{\mathrm{alg}}}})=0
$$
for $i\geq 1$ and $H^0(X_{E^{\mathrm{alg}}},\mathcal{O}_{X_{E^{\mathrm{alg}}}})\cong E^{\mathrm{alg}}$.
\end{proof}

In particular, we obtain that
$$
H^2_{\mathrm{fl}}(X_E,\mu_{p^r})\cong H^2_{\mathrm{fl}}(\Spec(E),\mu_{p^r})\cong \frac{1}{p^r}\Z/Z
$$
is canonically isomorphic to the $p^r$-torsion in the Brauer group $\mathrm{Br}(X_E)$ of $E$.

Finally, we record the calculation of $H^2_{\acute{e}t}(X_E,T)$ for an arbitrary torus $T$ over $E$ (cf.\ \cite[Th\'eor\'eme 2.7.]{fargues_g_torseurs_en_theorie_de_hodge_p_adique}).
 
\begin{theorem}
\label{theorem: second cohomology for tori}
Let $T$ be a torus over $E$. Then
$$
H^2_{\acute{e}t}(X_E,T)=H^2_{\mathrm{fl}}(X_E,T)=0.
$$
\end{theorem}
\begin{proof}
By \cite[Th\'eor\'eme (11.7)]{grothendieck_le_group_de_brauer_III} \'etale and flat cohomology for tori agree because tori are smooth. Let $E^\prime$ be a finite extension of $E$ splitting $T$. Then we obtain an exact sequence
$$
0\to T\to \mathrm{Res}_{E^\prime/E}T_{E^\prime}\to T^\prime\to 0
$$ 
of tori, where $\mathrm{Res}_{E^\prime/E}T$ denotes the Weil restriction of $T_{E^\prime}$ to $E$. By \Cref{theorem: fargues theorem} and \Cref{theorem: brauer group of the curve} we obtain a commutative diagram
$$
\xymatrix{
H^1_{\acute{e}t}(X_{E^\prime},T_{E^\prime})\ar[r]\ar[d]^\cong& H^1_{\acute{e}t}(X_E,T^\prime)\ar[r]\ar[d]^\cong& H^2_{\acute{e}t}(X_E,T)\ar[r] & H^2_{\acute{e}t}(X_{E^\prime},T_{E^\prime})=0 \\
B(\mathrm{Res}_{E^\prime/E}T_{E^\prime})\ar@{->>}[r] & B(T^\prime),
}
$$
where the lower horizontal arrow is surjective by Steinberg's theorem and \Cref{proposition: explicit description of b of g}.
This implies that 
$$
H^2_{\acute{e}t}(X_E,T)=0
$$
as desired.
\end{proof}

\Cref{theorem: second cohomology for tori} is related to Tate-Nakayama duality for tori (cf.\ \cite{tate_the_cohomology_groups_of_tori_in_finite_galois_extensions_of_number_fields}).
The pro\'etale covering $X_{\bar{E}}\to X_E$ yields a spectral sequence
$$
E^{ij}_2=H^i(\Gal(\bar{E}/E),H^j_{\acute{e}t}(X_{\bar{E}},T_{\bar{E}}))\Rightarrow H^{i+j}_{\acute{e}t}(X_E,T).
$$
But 
$$
H^1_{\acute{e}t}(X_{\bar{E}},T_{\bar{E}})\cong X_\ast(T)_{\Q}
$$ 
and 
$$
H^2_{\acute{e}t}(X_{\bar{E}},T_{\bar{E}})\cong 0,
$$ 
which implies by \Cref{theorem: second cohomology for tori} that there is an exact sequence
$$
H^1_{\acute{e}t}(X_E,T)\to H^0(\Gal(\bar{E}/E),X_\ast(T)_\Q)\to H^2(\Gal(\bar{E}/E),T(\bar{E}))\to 0
$$
where
$$
H^1_{\acute{e}t}(X_E,T)\cong B(T)\cong X_\ast(T)_\Gamma
$$
with $\Gamma=\Gal(\bar{E}/E)$ (cf.\ \cite{kottwitz_isocrystals_with_additional_structure_I}).

\section{Classification of reductive group schemes}

We continue with the notations from the previous two sections. 
We now want to classify reductive group schemes over the Fargues-Fontaine curve.
First we recall the definition of an affine group scheme $\mathbb{G}$ over $\varphi-\mathrm{Mod}_L$ (cf.\ \cite[Definition 9.1.8.]{dat_orlik_rapoport_period_domains}) as they provide the key examples of group schemes over $X_E$.

\begin{definition}
\label{definition: group scheme over isocrystals}
An affine group scheme $\mathbb{G}$ over $\varphi-\mathrm{Mod}_L$ is a Hopf algebra object $\mathcal{O}_{\mathbb{G}}$ in the category $\mathrm{Ind}-(\varphi-\mathrm{Mod}_L)$ of Ind-objects of $\varphi-\mathrm{Mod}_L$.
\end{definition}

In other words, an affine group scheme $\mathbb{G}$ over $\varphi-\mathrm{Mod}_L$ consists of an affine group scheme $G$ over $L$ and an isomorphism 
$$
\varphi_{\mathbb{G}}\colon\varphi^\ast_L G\cong G
$$ 
such that the Hopf algebra $\mathcal{O}_G$ of $G$ over $L$ is the increasing union of $\varphi_{\mathbb{G}}$-stable subspaces. Therefore we recover \cite[Definition 9.1.8.]{dat_orlik_rapoport_period_domains}. 
An affine group scheme $\mathbb{G}$ over $\varphi-\mathrm{Mod}_L$ is called reductive if the corresponding group scheme $G$ over $L$ is reductive.

If $G$ is an affine group scheme over $E$, then $G$ gives natural rise to an affine group scheme $\mathbb{G}$ over $\varphi-\mathrm{Mod}_L$ by considering the base change $G_L$ of $G$ to $L$ with its canonical isomorphism $\varphi^\ast_L(G_L)\cong G_L$. Equivalently, the Hopf algebra underlying $\mathbb{G}$ is given by $\mathcal{O}_G\otimes_EL$ with Frobenius acting on $L$.

Let $\mathbb{G}$ be an affine group scheme over $\varphi-\mathrm{Mod}_L$. Applying the tensor functor (cf.\ \Cref{theorem: classification of vector bundles})
$$
\mathcal{E}(-)\colon \varphi-\mathrm{Mod}_L\to \mathrm{Bun}_{X_E}
$$
to the Hopfalgebra $\mathcal{O}_{\mathbb{G}}$ yields a Hopf algebra 
$$
\mathcal{O}_{\mathcal{G}}:=\mathcal{E}(\mathcal{O}_{\mathbb{G}})
$$ 
over $\mathcal{O}_{X_E}$. Taking the relative Spec of this Hopf algebra 
$$
\mathcal{G}:=\underline{\Spec}(\mathcal{O}_{\mathcal{G}})
$$
defines a flat group scheme over $X_E$. We call $\mathcal{G}$ the group scheme (over $X_E$) associated with $\mathbb{G}$ and write 
$$
\mathcal{G}=\mathcal{E}(\mathbb{G})
$$ 
if we want to make this more precise. If $\mathbb{G}$ is reductive, then also $\mathcal{G}$ is reductive (over $X_E$) as this can be tested fiberwise and then over a large field extension of $E$.

We record the following general lemma.

\begin{lemma}
\label{lemma: inner forms}
Let $k$ be a field, let $G$ be an affine group scheme over $k$, let $G^\mathrm{ad}$ be its adjoint quotient and let $S$ be a $k$-scheme.
Let $\mathcal{Q}$ be a $G^\mathrm{ad}$-torsor over $S$ with corresponding inner form $\mathcal{G}_{\mathrm{Q}}$ of $G$ over $S$. Let 
$$
\omega_{\mathcal{Q}}\colon \Rep_k(G^{\mathrm{ad}})\to \Bun_{S}
$$ 
be the fiber functor of $\Rep_k(G^{\mathrm{ad}})$ over $S$ associated with $\mathcal{Q}$ (cf.\ \Cref{lemma: tannakian description of torsors}).
Then 
$$
\omega_{\mathcal{Q}}(\mathcal{O}_G)\cong \mathcal{O}_{\mathcal{G}_\mathcal{Q}}
$$
where $\mathcal{O}_G$ is considered as a Hopf algebra in the category of Ind-objects in $\mathrm{Rep}_k(G^\ad)$ via the adjoint action of $G^{\mathrm{ad}}$ and where $\mathcal{O}_{\mathcal{G}_\mathcal{Q}}$ denotes the $\mathcal{O}_S$-Hopf algebra of $\mathcal{G}_{\mathcal{Q}}$.  
\end{lemma}
\begin{proof}
  The fiber functor $\omega_{\mathcal{Q}}$ of $\Rep_k(G^\ad)$ is given by
$$
\omega_{\mathcal{Q}}\colon \Rep_k(G^\ad)\to \Bun_{S},\ V\mapsto \mathcal{Q}\times^{G^\ad}(V\otimes_k \mathcal{O}_S). 
$$
Moreover, the inner form $\mathcal{G}_{\mathcal{Q}}$ of $G$ is by definition given by the group scheme 
$$
\mathcal{G}_{\mathcal{Q}}=\mathcal{Q}\times^{G^\ad}G
$$ over $S$. Equivalently, this twisting can be done on the Hopf algebra of $G$. This shows 
$$
\omega_{\mathcal{Q}}(\mathcal{O}_G)\cong \mathcal{Q}\times^{G^\ad}\mathcal{O}_G\cong \mathcal{O}_{\mathcal{G}_{\mathcal{Q}}}
$$ 
and the proof is finished.
\end{proof}

We can now prove the following classification of reductive group schemes over the Fargues-Fontaine curve.

\begin{theorem}
\label{theorem: classification of reductive group schemes}
Let $\mathcal{G}$ be a reductive group scheme over the Fargues-Fontaine curve $X_E$.
Then there exists a reductive group scheme $\mathbb{G}$ over $\varphi-\mathrm{Mod}_L$, unique up to isomorphism, such that $\mathcal{G}$ is isomorphic to the group scheme $\mathcal{E}(\mathbb{G})$ associated with $\mathbb{G}$.
There exists moreover a quasi-split group $G^\prime$ over $E$ such that $\mathcal{G}$ is an inner form (over $X_E$) of $G^\prime$.
\end{theorem}
\begin{proof}
Let $G_0$ be the split reductive group over $E$ such that $\mathcal{G}$ is a form of $G_0$. 
Consider the canonical exact sequence
$$
1\to G^\ad_0\to \Aut(G_0)\to \mathrm{Out}(G_0)\to 1
$$
of group schemes over $X_E$. It gives rise to an exact sequence of pointed sets
$$
H^1_{\acute{e}t}(X_E,G^\ad_0)\to H^1_{\acute{e}t}(X_E,\Aut(G_0)) \xrightarrow{\delta} H^1_{\acute{e}t}(X_E,\mathrm{Out}(G_0)).
$$
But 
$$
H^1(\Gal(\bar{E}/E),\mathrm{Out}(G_0))\cong H^1_{\acute{e}t}(X_E,\mathrm{Out}(G_0))
$$ 
because $\mathrm{Out}(G_0)$ is constant and
$$
\pi_1^{\acute{e}t}(X_E)\cong \Gal(\bar{E}/E)
$$
by \Cref{theorem: etale fundamental group of the curve}.
The choice of a pinning of $G_0$ defines a splitting of $\delta$ and the image of 
$$
H^1(\Gal(\bar{E}/E),\mathrm{Out}(G_0))\cong H^1_{\acute{e}t}(X_E,\mathrm{Out}(G_0))
$$ under this splitting consists precisely of the classes of constant reductive group schemes 
$$
G^\prime\times_{\Spec(E)}X_E
$$ 
for $G^\prime$ a quasi-split form of $G_0$ over $E$.
Moreover, the elements in a fiber of $\delta$ are all inner forms of each other. In particular, we can see that $\mathcal{G}$ is an inner form (over $X_E$) of a quasi-split form $G^\prime$ (over $E$) of $G_0$.

Let $\mathcal{Q}$ be the $(G^\prime)^\ad$-torsor over $X_E$ and let 
$$
\omega_{\mathcal{Q}}\colon \Rep_E((G^\prime)^\ad)\to \mathrm{Bun}_{X_E}
$$
be the corresponding fiber functor. By \Cref{lemma: inner forms} we obtain
$$
\mathcal{O}_{\mathcal{G}}\cong \omega_{\mathcal{Q}}(\mathcal{O}_{G^\prime})
$$
where $\mathcal{O}_{G^\prime}$ is considered as a $(G^\prime)^\ad$-representation via the adjoint action.
By \Cref{theorem: fargues theorem} we know that the fiber functor $\omega_{\mathcal{Q}}$ is admissible and we can apply theorem \Cref{theorem: factorization of admissible fiber functors}. This yields an exact tensor functor
$$
\omega^\prime\colon \Rep_E((G^\prime)^\ad)\to \varphi-\mathrm{Mod}_L
$$
such that 
$$
\omega_{\mathcal{Q}}\cong \mathcal{E}(-)\circ \omega^\prime.
$$ 
In particular, we see that $\mathcal{G}$ is isomorphic to the group scheme associated with the group scheme $\mathbb{G}$ given by the Hopf algebra
$$
\omega^\prime(\mathcal{O}_{G^\prime})
$$
over $\varphi-\mathrm{Mod}_L$.
Moreover, as $\mathcal{G}$ is reductive, $\mathbb{G}$ is reductive.
As 
$$
\mathrm{gr}\circ \mathrm{HN}\circ \mathcal{E}(-)\cong \mathrm{Id}_{\varphi-\mathrm{Mod}_L}
$$ 
we see that $\mathbb{G}$ is determined, up to isomorphism, by $\mathcal{G}$.
\end{proof}

In short, the functor from reductive group schemes over $\varphi-\mathrm{Mod}_L$ to reductive group schemes over $X_E$ is faithful and induces a bijection on isomorphism classes. But it is not an equivalence, for $\mathcal{E}=\mathcal{O}\oplus \mathcal{O}(1)$ and $\mathcal{G}:=\mathrm{GL}(\mathcal{E})$ the global sections $\mathcal{G}(X_E)$ depend on the chosen algebraically closed perfectoid field $F/\F_q$. 

We can moreover obtain the following result.

\begin{lemma}
\label{lemma: finite extension yields pure inner form}
Let $\tilde{E}/E$ be an algebraic extension of $E$ such that $\tilde{E}$ has cohomological dimension $1$ and, if $E=\F_q((t))$, contains the perfection of $E$. Let $\mathcal{G}$ be a reductive group scheme over $X_E$. Then there exists a finite subextension $E^\prime\subseteq \tilde{E}$ such that the base change $\mathcal{G}_{X_E^\prime}$ is a pure inner form of a quasi-split group $G^\prime$ over $E^\prime$.
\end{lemma}
\begin{proof}
By \Cref{theorem: classification of reductive group schemes} we already know that $\mathcal{G}$ is an inner form of $G^\prime$ for a quasi-split group $G^\prime$ over $E$.
There exists an exact sequence (in the flat topology)
$$
1\to Z(G^\prime)\to G^\prime\to (G^\prime)^\ad\to 0
$$
yielding an exact sequence
$$
H^1_{\mathrm{fl}}(X_E,G^\prime)\to H^1_{\mathrm{fl}}(X_E,(G^\prime)^\ad)\to H^2_{\mathrm{fl}}(X_E,Z(G)). 
$$
Thus it suffices to show that every element in $H^2_{\mathrm{fl}}(X_E,Z(G^\prime))$ maps to zero under a finite extension contained in $\tilde{E}$. Let $T\subseteq G^\prime$ be a maximal torus and consider the exact sequence
$$
0\to Z(G^\prime)\to T\to \tilde{T}\to 0.
$$
Using the associated long exact sequence and \Cref{theorem: second cohomology for tori} it suffices to prove that
$$
\mathrm{Coker}(H^1_{\mathrm{fl}}(X_{\tilde{E}},T)\to H^1_{\mathrm{fl}}(X_E,\tilde{T}))=0.
$$
As $\tilde{E}$ has cohomological dimension $\leq 1$ (and contains the perfection of $E$ if $E=\F_q((t))$) the degrees of finite subextensions $E^\prime/E$ get divisible by arbitrary large integers $n\in \Z$.
We can conclude
$$
H^1_{\mathrm{fl}}(X_{\tilde{E}},T)\cong X_\ast(T)_{\Q}
$$
and
$$
H^1_{\mathrm{fl}}(X_{\tilde{E}},\tilde{T})\cong X_{\ast}(\tilde{T})_{\Q}.
$$
As $T\to \tilde{T}$ is surjective the morphism
$$
X_\ast(T)_{\Q}\to X_{\ast}(\tilde{T})_\Q
$$
is surjective. This finishes the proof of the lemma.
\end{proof}

For example, looking at the proof shows that one can take $\tilde{E}$ also to be the maximal unramified extension of $E$. But one can also take $\tilde{E}$ as the composition of a totally ramified extension and the perfection of $E$, if $E=\F_q((t))$.

We now start to classify torsors under reductive group schemes over $X_E$.

\begin{definition}
\label{definition: representations of group schemes over isocrystals}
Let $\mathbb{G}$ be a group scheme over $\varphi-\mathrm{Mod}_L$. We define $\Rep_E(\mathbb{G})$ as the category of representations of $\mathbb{G}$ on isocrystals, i.e., as the category of finite dimensional comodules of the Hopf algebra $\mathcal{O}_{\mathbb{G}}$ in the category $\varphi-\mathrm{Mod}_L$.
\end{definition}

In other words, an object $V\in \Rep_E(\mathbb{G})$ consists of an isocrystal $V\in \varphi-\mathrm{Mod}_L$ and a coaction $V\to V\otimes_L \mathcal{O}_{\mathbb{G}}$ which is moreover a morphism of Ind-isocrystals.
Clearly, the category $\Rep_E(\mathbb{G})$ is Tannakian over $E$.\footnote{We mention the following possible cause of confusion. If $G$ is a reductive group over $E$ with associated group scheme $\mathbb{G}$ over $\varphi-\mathrm{Mod}_L$, then the category of finite dimensional $E$-representations $\Rep_E(G)$ of $G$ is the full subcategory of $\Rep_E(\mathbb{G})$ given by representations of $\mathbb{G}$ whose underlying isocrystal is semistable of slope $0$.}
We denote by 
$$
\omega_{\mathrm{can}}^\prime\colon \Rep_E(\mathbb{G})\to \varphi-\mathrm{Mod}_L
$$
the canonical exact tensor functor sending a $\mathbb{G}$-representation to its underlying isocrystal. We define 
$$
\omega_{\mathrm{can}}:=\mathcal{E}(-)\circ \omega_{\mathrm{can}}^\prime\colon \Rep_E(\mathbb{G})\to \mathrm{Bun}_{X_E}, 
$$
and call it the canonical fiber functor of $\Rep_E(\mathbb{G})$ over $X_E$.

\begin{lemma}
\label{lemma: automorphisms of fiber functor}
Let $\mathbb{G}$ be a group scheme over $\varphi-\mathrm{Mod}_L$. Let 
$$
\omega\colon \Rep_E(\mathbb{G})\to \Bun_{S}
$$ 
be a fiber functor of $\Rep_E(\mathbb{G})$ over $S$ and let 
$$
\mathcal{G}:=\Aut^\otimes(\omega)
$$ 
be the corresponding group scheme over $S$. Let 
$$
\mathcal{O}_{\mathbb{G}}\in \mathrm{Ind}-\Rep_E(\mathbb{G})
$$ 
be the Hopf algebra underlying $\mathbb{G}$ considered as a representation of $\mathbb{G}$ via the adjoint action. Then there is a natural isomorphism 
$$
\mathcal{O}_{\mathcal{G}}\cong \omega(\mathcal{O}_{\mathbb{G}})
$$ 
of Hopf algebras. In other words, $\mathcal{O}_{\mathbb{G}}$ with the adjoint action by $\mathbb{G}$ is the fundamental group of $\Rep_E(\mathbb{G})$.
\end{lemma}
\begin{proof}
Let $\mathcal{R}$ be a quasi-coherent $\mathcal{O}_S$-algebra and let 
$$
f\colon \omega(\mathcal{O}_{\mathbb{G}})\to \mathcal{R}
$$ 
be a morphism of $\mathcal{O}_S$-algebras.
Then for $V\in \Rep_E(\mathbb{G})$ the composition
$$
\omega(V)\otimes_{\mathcal{O}_S}\mathcal{R}\to \omega(V)\otimes_{\mathcal{O}_S} \omega(\mathcal{O}_{\mathbb{G}})\otimes \mathcal{R}\to \omega(V)\otimes_{\mathcal{O}_S}\mathcal{R}, 
$$
where the first morphism is induced by the comultiplication of $V$, which is a morphism of $\mathbb{G}$-representations if $\mathcal{O}_{\mathbb{G}}$ is equipped with the adjoint action, and the second by $f$ and multiplication in $\mathcal{R}$,
is natural in $V$ and a naturally a tensor automorphism, hence defines a $\underline{\Spec}(\mathcal{R})$-valued point of $\mathcal{G}$. Conversely, if 
$$
\alpha_V\colon \omega(V)\otimes_{\mathcal{O_S}}\mathcal{R}\to \omega(V)\otimes_{\mathcal{O_S}} \mathcal{R}
$$
is a natural tensor automorphism, then evaluating $\alpha$ on the Ind-object $\mathcal{O}_{\mathbb{G}}$ of $\Rep_E(\mathbb{G})$ defines a morphism of $\mathcal{O}_S$-algebras
$$
\omega(\mathcal{O}_{\mathbb{G}})\to \omega(\mathcal{O}_{\mathbb{G}})\otimes_{\mathcal{O}_S} \mathcal{R}\xrightarrow{\alpha}\omega(\mathcal{O}_{\mathbb{G}})\otimes_{\mathcal{O}_S} \mathcal{R}\to \mathcal{R}, 
$$
where the first morphism is induced by the unit of $\mathcal{R}$ and the last by the counit of $\mathcal{O}_{\mathbb{G}}$, hence an $\mathcal{R}$-valued point of $\underline{\Spec}(\omega(\mathcal{O}_{\mathbb{G}}))$. One checks that both maps are inverse homomorphisms of groups. Hence, the claim follows.
\end{proof}

Applying this lemma to 
$$
\omega_{\mathrm{can}}\colon \Rep_E(\mathbb{G})\to \mathrm{Bun}_{X_E}
$$
shows that 
$$
\Aut^\otimes(\omega_{\mathrm{can}})\cong\mathcal{E}(\mathbb{G})
$$ 
is isomorphic to the group scheme associated with the group scheme $\mathbb{G}$ over $\varphi-\mathrm{Mod}_L$.
A similar proof shows moreover, that 
$$
\mathbb{G}\cong \Aut^\otimes(\omega^\prime_{\mathrm{can}})
$$ 
where $\Aut^\otimes(\omega^\prime_{\mathrm{can}})$ is the functor sending an algebra object $R$ in $\mathrm{Ind}-(\varphi-\mathrm{Mod}_L)$ to the group of tensor automorphisms of $\omega^\prime_{\mathrm{can}}\otimes R$. In particular, we can conclude that the band of the gerbe of the Tannakian category $\Rep_E(\mathbb{G})$ is reductive if $\mathbb{G}$ is.

We can now give a concrete classification of the category $\mathcal{B}(\Rep_E(\mathbb{G}))$ (cf.\ \Cref{definition: b of t}) of a connected group scheme $\mathbb{G}$ over $\varphi-\mathrm{Mod}_L$ (cf.\ \cite[Lemma 9.1.4]{dat_orlik_rapoport_period_domains}).
We let $\mathcal{B}(\mathbb{G})$ be the following category:
Its objects are elements $b\in \mathbb{G}(L)$ (more precisely, $b\in G(L)$ if $G$ is the group scheme over $L$ underlying $\mathbb{G}$) and the set of morphisms
from $b^\prime$ to $b$ are the elements $c\in \mathbb{G}(L)$ such that 
$$
b=cb^\prime\varphi_{\mathbb{G}}(c)^{-1}
$$ 
where 
$$
\varphi_{\mathbb{G}}\colon \mathbb{G}(L)\to \mathbb{G}(L)
$$ 
is induced by the given isomorphism $\varphi_{\mathbb{G}}$ of $G$ lying over the Frobenius $\varphi\colon L\to L$ of $L$.

Given an element $b\in \mathbb{G}(L)$ we can define an exact tensor functor
$$
\omega^\prime_b\colon \Rep_E(\mathbb{G})\to \varphi-\mathrm{Mod}_L
$$
by sending a $\mathbb{G}$-representation $V$ to the isocrystal
$$
(V,V\xrightarrow{\varphi_V}V\xrightarrow{b}V)
$$
where the second morphism denotes the action of $b\in \mathbb{G}(L)$ on $V$.

\begin{proposition}
\label{proposition: explicit description of b of g}
Let $\mathbb{G}$ be a connected affine group scheme over $\varphi-\mathrm{Mod}_L$. Sending $b\in \mathbb{G}(L)$ to $\omega_b^\prime$ defines an equivalence of categories
$$
\mathcal{B}(\mathbb{G})\cong \mathcal{B}(\Rep_E(\mathbb{G})).
$$
\end{proposition}
\begin{proof}
Let $\omega_{\mathrm{can}}^\prime\colon \Rep_E(\mathbb{G})\to \varphi-\mathrm{Mod}_L$
resp.\
$$
\mathrm{forg}\colon \varphi-\mathrm{Mod}_L\to \mathrm{Vec}_L
$$
be the canonical exact tensor functors.
If $G$ denotes the underlying connected group scheme of $\mathbb{G}$ over $L$ then by \Cref{lemma: automorphisms of fiber functor} 
$$
\Aut^\otimes(\mathrm{forg}\circ\omega^\prime_{\mathrm{can}})\cong G.
$$
Let 
$$
\omega^\prime\colon \Rep_E(\mathbb{G})\to \varphi-\mathrm{Mod}_L
$$ 
be an exact tensor functor. By Steinberg's theorem the functors $\mathrm{forg}\circ \omega^\prime$ and $\mathrm{forg}\circ \omega^\prime_{\mathrm{can}}$ are isomorphic. Let 
$$
\alpha\colon \mathrm{forg}\circ \omega^\prime\to \mathrm{forg}\circ \omega^\prime_{\mathrm{can}}
$$ 
be an $L$-linear tensor isomorphism.  Moreover, both functors come equipped with canonical $\varphi_L$-semilinear tensor automorphisms $\sigma^\prime,\sigma^\prime_{\mathrm{can}}$. Then 
$$
\beta:=\alpha\circ\sigma^\prime\circ \alpha^{-1}\circ \sigma^\prime_{\mathrm{can}}
$$ 
is an $L$-linear tensor automorphism of $\mathrm{forg}\circ \omega^\prime_{\mathrm{can}}$. Therefore, by the Tannakian formalism, there exists an element $b\in \mathbb{G}(L)$ such that $\beta$ is given by the multiplication by $b$. But then $\omega^\prime$ is isomorphic to $\omega^\prime_b$.
A similar reasoning shows that the functor $\mathcal{B}(\mathbb{G})\to \mathcal{B}(\Rep_E(\mathbb{G}))$ is fully faithful.   
\end{proof}

Let $G/E$ be a connected affine group scheme and let $\mathbb{G}$ be the connected affine group scheme over $\varphi-\mathrm{Mod}_L$. Using \cite[Lemma 9.1.4]{dat_orlik_rapoport_period_domains} (i.e., the same reasoning as in the proof above) we can conclude that there is a canonical chain of equivalences
$$
\mathcal{B}(\Rep_E(G))\cong \mathcal{B}(G)=\mathcal{B}(\mathbb{G})\cong \mathcal{B}(\Rep_E(\mathbb{G}))
$$
where $\mathcal{B}(G)$ is defined exactly as $\mathcal{B}(\mathbb{G})$. 

We will need the following lemma. If $\mathcal{T}$ is a Tannakian category over a field $k$ and $k^\prime/k$ a finite extension we denote by $\mathcal{T}_{k^\prime}$ the base change of $\mathcal{T}$ to $k^\prime$ (cf.\ \cite[Construction 2.12]{ziegler_graded_and_filtered_fiber_functors}).

\begin{lemma}
\label{lemma: base change of tannakian category}
  Let $E^\prime/E$ be a composition of a finite totally ramified and a finite purely inseparable extension and let $\mathbb{G}$ be an affine group scheme over $\varphi-\mathrm{Mod}_L$. Then the base change $\Rep_{E}(\mathbb{G})_{E^\prime}$ of $\Rep_{E}(\mathbb{G})$ to $E^\prime$ is equivalent to the category $\Rep_{E^\prime}(\mathbb{G}_{E^\prime})$ where $\mathbb{G}_{E^\prime}$ is the base change of $\mathbb{G}$ to $E^\prime$ (i.e., the Hopf algebra of $\mathbb{G}_{E^\prime}$ is given by the Hopf algebra $\mathcal{O}_{\mathbb{G}}\otimes_L L^\prime$ in $\varphi-\mathrm{Mod}_{L^\prime}$ where $L^\prime\cong E^\prime\otimes_E L$ is the completion of the maximal unramified extension of $E^\prime$).   
\end{lemma}
\begin{proof}
An object of $\Rep_E(\mathbb{G})_{E^\prime}$ is, by definition, given by a triple
$$
(V\in \varphi-\mathrm{Mod}_L, \alpha\colon E^\prime\mapsto \mathrm{End}_{\mathbb{G}}(V), c\colon V\to V\otimes_L \mathcal{O}_{\mathbb{G}})
$$
where $V$ is an isocrystal over $L$, $\alpha$ an $E$-algebra homomorphism and $c$ a coaction (in $\varphi-\mathrm{Mod}_L$) of $\mathcal{O}_{\mathbb{G}}$ on $V$. Moreover, $c$ is $E^\prime$-linear as $\alpha$ maps into the $\mathbb{G}$-endomorphisms of $V$.
We can map this triple to the $\mathbb{G}_{E^\prime}$-representation $(V,\tilde{c})$ where $V\in \varphi-\mathrm{Mod}_{L^\prime}$ is considered as an isocrystal over $L^\prime\cong E^\prime\otimes_E L$ with $E^\prime$ acting via $\alpha$ on $V$ and $\varphi_V\colon V\to V$ being the given $\varphi_L$-semilinear automorphism (which is also $\varphi_{L^\prime}$-linear). Here we have written $\varphi_L$ resp.\ $\varphi_{L^\prime}=\mathrm{Id}_{E^\prime}\otimes \varphi_L$ for the Frobenius of $L$ resp.\ $L^\prime$.
Finally, $\tilde{c}$ is defined as the composition
$$
V\xrightarrow{c} V\otimes_L\mathcal{O}_{\mathbb{G}}\cong V\otimes_{L^\prime}(L^\prime\otimes_L \mathcal{O}_{\mathbb{G}})
$$
which is a morphism of isocrystals over $L^\prime$.
Conversely, a $\mathbb{G}_{E^\prime}$-representation 
$$
(W\in \varphi-\mathrm{Mod}_{L^\prime},\tilde{c}\colon W\to W\otimes_{L^\prime}(L^\prime\otimes_L \mathcal{O}_{\mathbb{G}})
$$
can be sent to the triple
$$
(W, \alpha\colon E^\prime\to \mathrm{End}_{\mathbb{G}}(W),c\colon W\xrightarrow{\tilde{c}} W\otimes_L \mathcal{O}_{\mathbb{G}})
$$
where $W$ is considered as an isocrystal over $L$, $\alpha$ is the given action of $E^\prime$ on $W$ and $c$ the given coaction (using $W\otimes_{L^\prime}(L^\prime\otimes_L \mathcal{O}_{\mathbb{G}})\cong W\otimes_L \mathcal{O}_{\mathbb{G}}$).
These two functors define inverse equivalences of categories.  
\end{proof}

The same proof shows that for the Tannakian category $\Rep_k(G)$ of representations of an affine group scheme $G$ over a field $k$ and a finite extension $k^\prime$ of $k$ the base change $\Rep_k(G)_{k^\prime}$ is equivalent to the Tannakian category $\Rep_{k^\prime}(G_{k^\prime})$. 

\begin{theorem}
\label{theorem: classification of torsors general case}
Let $\mathcal{G}$ be a reductive group scheme over $X_E$ and let $\mathbb{G}$ be an affine group scheme over $\varphi-\mathrm{Mod}_L$ such that $\mathcal{G}$ is associated to $\mathbb{G}$, i.e., 
$$
\mathcal{G}\cong \mathcal{E}(\mathbb{G}).
$$ 
Then the canonical morphism
$$
B(\Rep_E(\mathbb{G}))\to H^1_{\acute{e}t}(X_E,\mathcal{G})
$$
is bijective.
\end{theorem}
\begin{proof}
Let $G^\prime$ be a quasi-split group over $E$ such that $\mathcal{G}$ is an inner form of $G^\prime$ (cf.\ \Cref{theorem: classification of reductive group schemes}).
We first assume that $\mathcal{G}$ is a pure inner form of $G^\prime$. Then there exists a $G^\prime$-torsor $\mathcal{Q}$ over $X_E$ such that 
$$
\mathcal{G}\cong \Aut_{G^\prime}(\mathcal{Q}).
$$
By \Cref{theorem: fargues theorem}, \Cref{proposition: explicit description of b of g} and \Cref{lemma: automorphisms of fiber functor} we can see that there exists an element $b\in G^\prime(L)$ such that 
$$
\mathcal{O}_{\mathbb{G}}\cong\mathcal{O}_{G^\prime}\otimes_EL
$$ 
with $\varphi_{\mathbb{G}}$ given by the composition
$$
\mathcal{O}_{G^\prime}\otimes_EL\xrightarrow{\mathrm{Id}\otimes\varphi_L}\mathcal{O}_{G^\prime}\otimes_EL\xrightarrow{\mathrm{Ad}(b)}\mathcal{O}_{G^\prime}\otimes_EL
$$
where 
$$
\varphi_L\colon L\to L
$$ 
denotes the Frobenius on $L$ and $\mathrm{Ad}(b)$ be adjoint action of $b$ on $\mathcal{O}_{G^\prime}\otimes_E L$.
Let $\mathbb{G}^\prime$ be the group scheme over $\varphi-\mathrm{Mod}_L$ associated with $G^\prime$.
One can deduce that 
$$
\Rep_E(\mathbb{G})\cong \Rep_E(\mathbb{G}^\prime)
$$ 
by mapping a representation $V$ of $\mathbb{G}^\prime$ to the $\mathbb{G}^\prime$-representation $bV$ (cf.\ \cite[Example 9.1.22]{dat_orlik_rapoport_period_domains}).
In particular, we can conclude that every fiber functor 
$$
\omega\colon \Rep_E(\mathbb{G})\to \mathrm{Bun}_{X_E}
$$ 
for $\Rep_E(\mathbb{G})$ is admissible because this holds true for $\Rep_E(\mathbb{G}^\prime)$ by \Cref{theorem: fargues theorem} (and the fact these fiber functors identify with fiber functors for $\Rep_E(G^\prime)$, cf.\ \Cref{proposition: explicit description of b of g}).
By \Cref{theorem: classification of admissible fiber functors} we can conclude in this case.

Now assume that $\mathcal{G}$ is an arbitrary inner form of $G^\prime$. 
Let $E^\prime$ be a composition of a finite totally ramified extension of $E$ and a finite purely inseparable extension of $E$ such that $\mathcal{G}_{X_{E^\prime}}$ is a pure inner form of $G^\prime_{E^\prime}$ (the existence of such a field extension is guaranteed by \Cref{lemma: finite extension yields pure inner form} and the discussion following it).
We already know (from \Cref{theorem: classification of admissible fiber functors}) that the map
$$
B(\Rep_E(\mathbb{G}))\to H^1_{\acute{e}t}(X_E,\mathcal{G})
$$
is injective and, again by \Cref{theorem: classification of admissible fiber functors}, it suffices to check that every fiber functor
$$
\omega\colon \Rep_E(\mathbb{G})\to \mathrm{Bun}_{X_E}
$$
is admissible.
Let $\omega\colon \Rep_E(\mathbb{G})\to \mathrm{Bun}_{X_E}$ be such a fiber functor. By \cite[Construction 2.12]{ziegler_graded_and_filtered_fiber_functors} the composition
$$
\Rep_E(\mathbb{G})\to \mathrm{Bun}_{X_E}\to \mathrm{Bun}_{X_{E^\prime}}
$$
factors over a fiber functor 
$$
\omega_{E^\prime}\colon \Rep_E(\mathbb{G})_{E^\prime}\to \mathrm{Bun}_{X_{E^\prime}},
$$
where $\Rep_E(\mathbb{G})_{E^\prime}$ denotes the base change of the category $\Rep_E(\mathbb{G})$ to $E^\prime$.
By \Cref{lemma: base change of tannakian category} the categories $\Rep_E(\mathbb{G})_{E^\prime}$ and $\Rep_{E^\prime}(\mathbb{G}_{E^\prime})$ are equivalent. As $\mathcal{G}_{X_{E^\prime}}$ is a pure inner form of $G^\prime_{E^\prime}$ (and associated with $\mathbb{G}_{E^\prime}$) the fiber functor $\omega_{E^\prime}$ is admissible by the case already proven. By \Cref{lemma: admissible fiber functors and finite base change} this implies that $\omega$ is admissible. In particular, this finishes the proof of the classification of $\mathcal{G}$-torsors.
\end{proof}

\section{Uniformization results}
\label{section: uniformization}

In this section we establish uniformization results for $\mathcal{G}$-torsors over the Fargues-Fontaine curve.

First, we want to prove that reductive group schemes over the Fargues-Fontaine curve become constant after removing a closed point.

We will need the following lemma.

\begin{lemma}
\label{lemma: reduction to tori in the quasi-split case}
Let $G/E$ be a quasi-split reductive group over $E$. Then every element in $B(G)$ admits a reduction to some torus $T\subseteq G$.  
\end{lemma}
\begin{proof}
For $E/\mathbb{Q}_p$ this is \cite[Proposition 7.2.]{fargues_g_torseurs_en_theorie_de_hodge_p_adique}.
Using \cite[Proposition 13.1]{kottwitz_bg_for_all_local_and_global_fields} we can extend this argument once we establish that in the case $E=\F_q((t))$ every element in $B(G)$ admits a reduction to some basic element in $B(M)$ for some Levi subgroup $M\subseteq G$. Let $b\in B(G)$ and consider the filtered fiber functor
$$
\omega\colon \Rep_E(G)\xrightarrow{b} \varphi-\mathrm{Mod}_L\to \Bun_{X_E}\xrightarrow{HN} \mathrm{Fil}^\Q\Bun_{X_E}.
$$
It defines a parabolic subgroup scheme 
$$
P(\omega)\subseteq G\times_EX_E.
$$ 
If $B\subseteq G$ is a Borel subgroup, then there exists a unique standard parabolic 
$$
B\subseteq P\subseteq G
$$ 
such that for every $x\in X_E$ the groups $P(\omega)_{\bar{x}}$ and $P$ are conjugated. Using \cite[Expos\'e XXVI, Proposition 1.3.]{sgaIII_schemas_en_groupes_expose_23_a_26} we can conclude that there exists a $P$-torsor $\mathcal{Q}$ over $X_E$ such that 
$$
P(\omega)\cong P\times^{P}\mathcal{Q}
$$ 
with $P$ acting on $P$ via conjugation (in \cite[Section 5.1.]{fargues_g_torseurs_en_theorie_de_hodge_p_adique} this $P$-torsor would be called the canonical reduction of the torsor associated with $P$). As in \cite[Proposition 5.16.]{fargues_g_torseurs_en_theorie_de_hodge_p_adique} this $P$-torsor admits a reduction to a Levi subgroup $M\subseteq P$ because for $\lambda\geq 0$ the vector bundles 
$$
\mathrm{gr}^\lambda(U(\omega)) 
$$
(cf.\ \Cref{definition: group sheaves associated to filtered fiber functors})
are semistable of slope $\lambda\geq 0$.
By construction this $M$-torsor is semistable which implies that $b$ reduces to a basic element in $B(M)$ (cf.\ \cite[Proposition 5.12.]{fargues_g_torseurs_en_theorie_de_hodge_p_adique}).   
\end{proof}

As in \cite{fargues_g_torseurs_en_theorie_de_hodge_p_adique} this implies the following theorem.

\begin{theorem}
\label{theorem: uniformization constant quasi-split case}
Let $G/E$ be a quasi-split reductive group over $E$ and let $x\in X_E$ be a closed point. Then every $G$-torsor $\mathcal{Q}$ over $X_E$ is trivial over $X_E\setminus \{x\}$.
\end{theorem}
\begin{proof}
Using \Cref{lemma: reduction to tori in the quasi-split case} and \Cref{theorem: second cohomology for tori} the same proof as in \cite[Th\'eor\`eme 7.1.]{fargues_g_torseurs_en_theorie_de_hodge_p_adique} works.  
\end{proof}

\begin{lemma}
\label{lemma: reductive group schemes over punctured curve}
Let $\mathcal{G}$ be a reductive group scheme over $X_E$ and let $x\in X_E$ be a closed point. Then $\mathcal{G}_{|X_E-\{x\}}$ is isomorphic to a constant quasi-split reductive group.
\end{lemma}
\begin{proof}
By \Cref{theorem: classification of reductive group schemes} there exists a quasi-split group $G^\prime$ over $E$ such that $\mathcal{G}$ is an inner form of $G^\prime$. Let $\mathcal{Q}$ be a $(G^\prime)^\ad$-torsor such that 
$$
\mathcal{G}\cong \mathcal{Q}\times^{(G^\prime)^\ad}G^\prime.
$$
Then $(G^\prime)^\ad$ is again quasi-split as the image of a Borel subgroup $B\subseteq G^\prime$ is a Borel subgroup of $(G^\prime)^\ad$. By \Cref{theorem: uniformization constant quasi-split case} every $(G^\prime)^\ad$-torseur over $X_E$ is trivial over the punctured curve $X_E-\{x\}$. In particular, $\mathcal{Q}_{X_E-\{x\}}$ is trivial which shows that 
$$
\mathcal{G}_{|X_E-\{x\}}\cong G^\prime\times_{\Spec(E)}(X_E-\{x\})
$$ 
is isomorphic to a constant quasi-split reductive group.  
\end{proof}

In fact, reductive group schemes over $X_E$ are not too far from being constant.
Let $\mathcal{G}$ be a reductive group scheme over $X_E$ and write (cf.\ \Cref{theorem: classification of reductive group schemes})
$$
\mathcal{G}\cong\mathcal{Q}\times^{G^\ad}G
$$
for a quasi-split reductive group $G/E$ over $E$ and a $G^\ad$-torsor $\mathcal{Q}$ over $X_E$. Let $b\in G^\ad(L)$ be an element giving rise to $\mathcal{Q}$ (cf.\ \Cref{proposition: explicit description of b of g}).
Assume first that $b$ is basic (cf.\ \cite[5.1.]{kottwitz_isocrystals_with_additional_structure_I}). We claim that in this case $\mathcal{G}$ is already constant. Namely, the group $G^\ad$ is semisimple and it suffices to prove the following lemma showing that $\mathcal{Q}$ is isomorphic to the pullback of a $G^\ad$-torsor over $E$ in this case. We denote by $B(H)_{\mathrm{basic}}$ the set of basic elements in $B(H)$ if $H$ is a reductive group over $E$ (cf.\ \cite[5.1.]{kottwitz_isocrystals_with_additional_structure_I}).

\begin{lemma}
\label{lemma: basic elements for semisimple groups}
Let $H/E$ be a semisimple group. Then the canonical map
$$
H^1(\Gal(\bar{E}/E),H)\to B(H)_{\mathrm{basic}}
$$
(cf.\ \cite[1.8.]{kottwitz_isocrystals_with_additional_structure_I}) is bijective.
\end{lemma}
\begin{proof}
By \cite[4.5.]{kottwitz_isocrystals_with_additional_structure_I} an element $c\in B(H)$ lies in the image of the canonical injection
$$
H^1(\Gal(\bar{E}/E),H)\to B(H)
$$
if and only if the associated morphism $\nu_c\colon \mathbb{D}\to H$ is trivial where $\mathbb{D}$ is the constant pro-torus with character group $\mathbb{Q}$ (cf.\ \cite[4.2.]{kottwitz_isocrystals_with_additional_structure_I} for the construction of $\nu_c$). Moreover, by definition $c\in B(H)$ is basic if and only if $\nu_c$ factors through the center of $H$ (\cite[5.1.]{kottwitz_isocrystals_with_additional_structure_I}). But if $H$ is semisimple the center of $H$ is finite which implies that every central homomorphism $\nu_c\colon \mathbb{D}\to H$ must be trivial as $\mathbb{D}$ is connected.
\end{proof}

Now assume that $b\in G^\ad(L)$ is arbitrary. By \cite[Proposition 6.2.]{kottwitz_isocrystals_with_additional_structure_I} resp.\ the proof of \Cref{lemma: reduction to tori in the quasi-split case} the element $b$ is $\varphi_L$-conjugate to some $b^\prime\in M^\prime(L)$ for some Levi subgroup $M^\prime\subseteq G^\ad$ such that $b^\prime$ is basic. In particular,
$$
\mathcal{G}\cong \mathcal{Q}\times^{G^\ad}G\cong \mathcal{Q}^\prime \times^{M^\prime} G
$$  
if $\mathcal{Q}^\prime$ denotes the $M^\prime$-torsor corresponding to $b^\prime$ (cf.\ \Cref{proposition: explicit description of b of g}).
Let $M\subseteq G$ be the preimage of $M^\prime\subseteq G^\ad$ under the canonical map $G\to G^\ad$. 
As $M^\prime$ is connected the image $M^{\prime\prime}$ of $M^\prime$ in $\Aut(M)$ must be contained in the adjoint group $M^\ad$ of $M$. Let $\mathcal{Q}^{\prime\prime}$ be the push forward of the $M^\prime$-torsor $\mathcal{Q}^\prime$ to $M^{\prime\prime}$. The element $c^{\prime\prime}\in B(M^{\prime\prime})$ corresponding to $\mathcal{Q}^{\prime\prime}$ will again be basic. Namely, if $c^\prime\in B(M^\prime)$ corresponds to the $M^\prime$-torseur $\mathcal{Q}^\prime$ over $X_E$, then
$$
\nu_{c^{\prime\prime}}=\alpha\circ \nu_{c^\prime}=1
$$
because the homomorphism $\nu_{c^\prime}\colon \mathbb{D}\to M^\prime$ factors through the center $Z(M^\prime)$ of $M^\prime$ and
$$
Z(M^\prime)\subseteq \mathrm{Ker}(M^\prime\xrightarrow{\alpha} M^{\prime\prime}).
$$
By \Cref{lemma: basic elements for semisimple groups} we can conclude
that $\mathcal{G}$ contains the constant reductive group
$$
\mathcal{Q}^{\prime\prime}\times^{M^{\prime\prime}} M\cong\mathcal{Q}^{\prime}\times^{M^{\prime}} M\subseteq \mathcal{Q}^{\prime}\times^{M^{\prime}} G\cong \mathcal{G}. 
$$
which is of the same rank.

In general reductive group schemes can be non-constant. For example, let $\mathcal{E}$ be a non-semistable vector bundle on $X_E$. Then the reductive group scheme
$$
\mathcal{G}:=\mathrm{GL}(\mathcal{E})
$$
is non-constant.

Finally we record the following ``uniformization result'' generalizing \cite[Th\'eo\`eme 7.1.]{fargues_g_torseurs_en_theorie_de_hodge_p_adique} (resp.\ \Cref{theorem: uniformization constant quasi-split case}).

\begin{theorem}
\label{theorem: uniformization}
Let $\mathcal{G}$ be a reductive group scheme over $X_E$, let $x\in X_E$ be a closed point and let $\mathcal{Q}$ be a $\mathcal{G}$-torsor. Then $\mathcal{Q}_{X_E-\{x\}}$ is isomorphic to the trivial $\mathcal{G}_{X_E-\{x\}}$-torsor.
\end{theorem}                
\begin{proof}
By \Cref{lemma: reductive group schemes over punctured curve} the group scheme 
$$
\mathcal{G}_{|X_E-\{x\}}\cong G\times_{\Spec(E)}(X_E-\{x\})
$$ 
is isomorphic to a constant quasi-split reductive group scheme. Let $\mathcal{Q}^\prime$ be the $G\times_{\Spec(E)}(X_E-\{x\})$-torseur corresponding to $\mathcal{Q}_{X_E-\{x\}}$ under such an isomorphism. Then $\mathcal{Q}^\prime$ admits an extension to a $G$-torsor over $X_E$. Namely, by Beauville-Laszlo glueing (cf.\ \cite{beauville_laszlo_un_lemme_de_descente} and \Cref{lemma: tannakian description of torsors}) it suffices to construct a $G$-torsor over $\Spec(\widehat{\mathcal{O}_{X_E,x}})$ together with an isomorphism to $\mathcal{Q}^\prime$ over $\Spec(\mathrm{Frac}(\widehat{\mathcal{O}_{X_E,x}}))$. But abstractly 
$$
\widehat{\mathcal{O}_{X_E,x}}\cong k[[t]]
$$ 
is isomorphic to a power series ring by Cohen's structure theorem. In particular, 
$$
\mathrm{Frac}(\widehat{\mathcal{O}_{X_E,x}})\cong k((t))
$$ 
is of cohomological dimension $1$ as $k$ is algebraically closed in our case. Steinberg's theorem finally implies that every $G$-torsor over $\Spec(k((t)))$ is trivial and, in particular, admits an extension to $k[[t]]$.
Thus, let $\mathcal{Q}^{\prime\prime}$ be a $G$-torsor extending $\mathcal{Q}^\prime$. By \Cref{theorem: uniformization constant quasi-split case} we can conclude that 
$$
\mathcal{Q}^{\prime\prime}_{|X_E-\{x\}}\cong \mathcal{Q}^\prime\cong \mathcal{Q}_{X_E-\{x\}}
$$
is isomorphic to the trivial torsor because $G^\prime$ is quasi-split.  
\end{proof}

In fact, in \Cref{theorem: uniformization} we have shown
$$
H^1_{\acute{e}t}(X_E-\{x\},G)=0
$$
for every reductive group $G$ over $E$.

\end{document}